\documentclass[12pt,reqno]{amsart}
\usepackage[a4paper]{geometry}

\usepackage{amssymb,amsmath,amsthm,enumerate,bbm}
\usepackage{graphicx}
\usepackage{verbatim}

\usepackage{color}

\usepackage[
hypertexnames=false, colorlinks, citecolor=red, linkcolor=red]{hyperref} 

\DeclareMathOperator{\Dom}{Dom}
\DeclareMathOperator{\Ran}{Ran}
\DeclareMathOperator{\Ker}{Ker}

\DeclareMathOperator{\supp}{supp}

\DeclareMathOperator{\Tr}{Tr}
\DeclareMathOperator{\sech}{sech}

\DeclareMathOperator{\Span}{span}

\newcommand{\abs}[1]{\lvert#1\rvert}
\newcommand{\Abs}[1]{\left\lvert#1\right\rvert}
\newcommand{\norm}[1]{\lVert#1\rVert}

\newcommand{\jap}[1]{\langle#1\rangle}

\newcommand{\comp}{{\mathrm{comp}}}

\newcommand{\bbR}{{\mathbb R}}
\newcommand{\bbC}{{\mathbb C}}
\newcommand{\bbZ}{{\mathbb Z}}

\newcommand{\calL}{\mathcal{L}}
\newcommand{\calM}{\mathcal{M}}

\newcommand{\calD}{\mathcal{D}}

\newcommand{\calH}{\mathcal{H}}

\newcommand{\1}{\mathbf{1}}

\newcommand{\dd}{\mathrm d}

\newcommand{\finite}{\text{\rm finite}}
\newcommand{\cofinite}{\text{\rm cofinite}}

\numberwithin{equation}{section}

\theoremstyle{plain}
\newtheorem{theorem}{\bf Theorem}[section]
\newtheorem*{theorem*}{Theorem}
\newtheorem{lemma}[theorem]{\bf Lemma}
\newtheorem{proposition}[theorem]{\bf Proposition}
\newtheorem*{proposition*}{\bf Proposition}

\newtheorem{corollary}[theorem]{\bf Corollary}

\theoremstyle{definition}
\newtheorem{definition}[theorem]{\bf Definition}
\newtheorem*{definition*}{\bf Definition}

\theoremstyle{remark}
\newtheorem*{remark*}{\bf Remark}
\newtheorem{remark}[theorem]{\bf Remark}
\newtheorem{example}[theorem]{\bf Example}
\newtheorem*{example*}{\bf Example}

\makeatletter
\newcommand*\bigcdot{\mathpalette\bigcdot@{.5}}
\newcommand*\bigcdot@[2]{\mathbin{\vcenter{\hbox{\scalebox{#2}{\,\,$\m@th#1\bullet$\,\,}}}}}
\makeatother

\newcommand{\ci}[1]{{\vphantom{\rule[-0.65ex]{0ex}{0.35ex}}}_{#1}}
\newcommand{\ti}[1]{_{\scriptstyle \text{\rm #1}}}

\begin{document}

\title[Inverse spectral problems for positive Hankel operators]{Inverse spectral problems for positive Hankel operators}

\author{Alexander Pushnitski}
\address{Department of Mathematics, King's College London, Strand, London, WC2R~2LS, U.K.}
\email{alexander.pushnitski@kcl.ac.uk}

\author{Sergei Treil}
\address{Department of Mathematics, Brown University, Providence, RI 02912, USA}
\email{treil@math.brown.edu}
\thanks{Work of S.~Treil is 
supported  in part by the National Science Foundation under the grants   
DMS-2154321, DMS-2452407}



\date{\today}

\begin{abstract}
A Hankel operator $\Gamma$ in $L^2(\mathbb{R}_+)$ is an integral operator with the integral kernel of the form $h(t+s)$, where $h$ is known as the kernel function. It is known that $\Gamma$ is positive semi-definite if and only if $h$ is the Laplace transform of a positive measure $\mu$ on $\mathbb{R}_+$. Thus, positive semi-definite Hankel operators $\Gamma$ are parameterised by measures $\mu$ on $\mathbb{R}_+$. We consider the class of $\Gamma$ corresponding to \emph{finite} measures $\mu$. In this case it is possible to define the (scalar) spectral measure $\sigma$ of $\Gamma$ in a natural way. The measure $\sigma$ is also finite on $\mathbb{R}_+$. This defines the \emph{spectral map} $\mu\mapsto\sigma$ on finite measures on $\mathbb{R}_+$. We prove that this map is an involution; in particular, it is a bijection.  We also consider a dual variant of this problem for measures $\mu$ that are not necessarily finite but have the finite integral
\[
\int_0^\infty x^{-2}\mathrm{d}\mu(x);
\]
we call such measures \emph{co-finite}.
\end{abstract}

\maketitle

\section{Introduction}\label{sec.z}

\subsection{Hankel operators: discrete and continuous  realizations}
We denote 
\[
\bbZ_+=\{0,1,2,\dots\}\text{ and }\bbR_+=(0,\infty)
\text{ and write }
\ell^2:=\ell^2(\bbZ_+)\text{ and }L^2:=L^2(\bbR_+).
\]
Hankel operators can be defined in two equivalent ways: as infinite matrices acting on $\ell^2$ 
(discrete realization) or as integral operators acting on $L^2$ (continuous realization). For a 
sequence $\{h_j\}_{j\in\bbZ_+}$ of complex numbers, a Hankel matrix is the operator on $\ell^2$ 
that can be identified with the matrix  
\begin{equation}
\{h_{k+j}\}_{j,k\in\bbZ_+}
\label{zz0}
\end{equation}
with respect to the canonical basis in $\ell^2$. Similarly, if $h:\bbR_+\to\bbC$ is a function, then the Hankel operator $\Gamma_h$ is the operator on $L^2$, defined by 
\[
(\Gamma_h f)(t)=\int_0^\infty h(t+s)f(s)\dd s, \quad t\in\bbR_+,
\]
on a suitable set of functions $f\in L^2$. The function $h$ is called \emph{the kernel function}. 

The focus of this paper is the \emph{continuous} realization of Hankel operators. 

\subsection{Positive semi-definiteness and boundedness}
It is a classical theorem of Hamburger that a Hankel matrix \eqref{zz0} is positive semi-definite if and only if $h_j$ are moments of a positive measure $\mu$ on $\bbR$:
\[
h_j=\int_\bbR x^j \dd\mu(x), \quad j\geq0.
\]
Similarly, it is known that an integral Hankel operator $\Gamma_h$ is positive semi-definite if and only if $h$ can be represented as the Laplace transform of a positive measure $\mu$ on $\bbR$:
\begin{equation}
h(t)=\int_\bbR e^{-tx}\dd\mu(x), \quad t>0,
\label{z2}
\end{equation}
where $\mu$ is such that the integral converges for all $t>0$. The precise statement and a proof in a very general context can be found in \cite[Theorems 5.1 and 5.3]{Ya2}. We will write $h_\mu$ instead of  $h$ for the kernel function \eqref{z2} and we will write $\Gamma_\mu$ instead of $\Gamma_{h_\mu}$. 

A positive semi-definite Hankel operator $\Gamma_\mu$ is bounded if and only if 
\[
\supp\mu\subset[0,\infty)\quad\text{ and }\quad\mu(\{0\})=0
\]
and the Carleson condition
\begin{equation}
\mu((0,a))\leq Ca, \quad \forall a>0
\label{z3}
\end{equation}
is satisfied (see e.g. \cite[Appendix]{PuTreil}). 

We recall two classical examples.
\begin{example}
Let $\mu$ be the Lebesgue measure restricted to $\bbR_+$. Then $h_\mu(t)=1/t$ and $\Gamma_\mu$ is known as the \emph{Carleman operator}. It is well-known  \cite[Section 10.2]{Peller} that the Carleman operator is bounded but not compact and has a purely absolutely continuous spectrum $[0,\pi]$ with multiplicity two. More precisely, the Carleman operator is unitarily equivalent to the orthogonal sum of two copies of the operator of multiplication by the independent variable in $L^2(0,\pi)$.  
\end{example}

\begin{example}
Let $\mu$ be a single point mass at $\alpha>0$ with the total mass $=1$. Then $h_\mu(t)=e^{-\alpha t}$ and $\Gamma_\mu$ is a rank one Hankel operator, which (up to a factor) coincides with the projection onto the function $e^{-\alpha t}$ in $L^2(\bbR_+)$. Of course, if $\mu$ is a finite linear combination of point masses, then $\Gamma_\mu$ is a finite rank operator. 
\end{example}

In this paper, we work with Hankel operators $\Gamma_\mu$ for measures $\mu$ that do not necessarily satisfy the Carleson condition. Thus, our Hankel operators $\Gamma_\mu$ are \emph{not necessarily bounded}. The definition of such Hankel operators requires some care; in the next subsection we address this question. 

\subsection{The class $\calM$ and the definition of $\Gamma_\mu$}\label{sec.z5}

\begin{definition*}[Class $\calM$]
For a positive measure $\mu$ on $\bbR_+$, we will write $\mu\in\calM$, if $\mu(\{0\})=0$ and the Laplace transform of $\mu$ is finite, i.e. 
\[
h_\mu(t):=\int_0^\infty e^{-\lambda t}\dd\mu(\lambda)<\infty, \quad \forall t>0.
\]
\end{definition*}
For $\mu\in\calM$, we have $h_\mu(t)\to0$ as $t\to\infty$, but the rate of decay may be arbitrarily slow.  The function $h_\mu$ is infinitely smooth on $\bbR_+$ but may have a singularity as $t\to0$.

For $\mu\in\calM$, we can write the quadratic form corresponding to $\Gamma_\mu$ as follows:
\begin{align*}
\int_0^\infty \int_0^\infty h_\mu(t+s) f(s)\overline{f(t)}\dd s\,  \dd t
&=\int_0^\infty \int_0^\infty \left\{\int_0^\infty e^{-tx}e^{-sx}\dd\mu(x)\right\} f(s)\overline{f(t)}\dd s\,  \dd t
\\
&=\int_0^\infty \Abs{\int_0^\infty e^{-tx}f(t)\dd t}^2\dd\mu(x)
\end{align*}
on a suitable set of ``nice'' functions $f$. Motivated by this, let us denote by $\calL f$ the Laplace transform of $f$, 
\[
\calL f(x)=\int_0^\infty e^{-tx}f(t)\dd t, \quad\quad x>0,
\]
and define the quadratic form 
\[
Q_\mu[f]=\int_0^\infty\abs{\calL f(x)}^2\dd\mu(x),
\]
with the domain 
\[
\Dom Q_\mu=\{f\in L^2: Q_\mu[f]<\infty\}. 
\]
Below we denote by $C^\infty_\comp(\bbR_+)$ the set of infinitely smooth functions compactly supported in $\bbR_+$ (in particular, the support is separated away from the origin). 
\begin{proposition}\cite{Ya1,PuTreil}\label{thm.z1}
For any $\mu\in\calM$, the quadratic form $Q_\mu$ is closed. 
The set $C^\infty_\comp(\bbR_+)$ belongs to $\Dom Q_\mu$ and is dense in $\Dom Q_\mu$ with respect to the norm induced by $Q_\mu$. 
\end{proposition}
This theorem was proved in \cite[Theorem~3.10]{Ya1} under the additional assumption 
\begin{equation}
\int_0^\infty (1+\lambda)^{-k}\dd\mu(\lambda)<\infty 
\quad\text{ for some $k>0$.}
\label{b0a}
\end{equation}
This assumption was lifted in \cite{PuTreil}. 
Following Yafaev \cite{Ya1}, we accept
\begin{definition}[Operator $\Gamma_\mu$]\label{def:1}
Let $\mu\in\calM$. We denote by $\Gamma_\mu$ the self-adjoint operator in $L^2$ corresponding to the quadratic form $Q_\mu$. 
\end{definition}
The precise description of the domain of $\Gamma_\mu$ is a separate non-trivial problem that we do not discuss in this paper; see \cite{PuTreil}.

\subsection{The kernel and the non-zero part of \texorpdfstring{$\Gamma_\mu$}{Gamma_mu}}
In order to state our results, we need the following simple fact from \cite{PuTreil}. 

\begin{proposition}\label{prp.z1d}\cite{PuTreil}
Let $\mu\in\calM$. The kernel of $\Gamma_\mu$ is either trivial or infinite-dimensional. 
It is infinite-dimensional if and only if $\mu$ is a purely atomic measure supported on a sequence of points $\{a_k\}_{k\in\bbZ_+}\subset \bbR_+$ satisfying the Blaschke condition \begin{equation}
\sum_{k\in\bbZ_+} \frac{a_k}{a_k^2+1}<\infty.
\label{z0a}
\end{equation}
\end{proposition}
Below we are interested in the \emph{non-zero part of $\Gamma_\mu$} defined as the restriction
\[
\boxed{\Gamma_\mu^{\perp}:=\Gamma_\mu|_{(\Ker \Gamma_\mu)^\perp}.}
\]

\subsection{The classes $\calM_\finite$ and $\calM_\cofinite$}
Let $\calM_\finite\subset\calM$ be the set of all \emph{finite} measures.  
Next, let $\calM_\cofinite\subset\calM$ be the set of measures satisfying 
\[
\int_0^\infty \frac{\dd\mu(x)}{x^2}<\infty;
\]
we will call such measures \emph{co-finite} for short. 
Observe that for $\mu\in\calM_\finite$, the Laplace transform $h_\mu$ is bounded:
\begin{equation}
\mu\in\calM_\finite\quad \Rightarrow \quad h_\mu(t)\leq \mu(\bbR_+),
\label{eq:2}
\end{equation}
and for $\mu\in\calM_\cofinite$, the Laplace transform $h_\mu$ has a fast decay at infinity:
\begin{equation}
\mu\in\calM_\cofinite\quad \Rightarrow \quad h_\mu(t)\leq\frac{C}{t^2}\int_0^\infty\frac{\dd\mu(x)}{x^2}
\label{eq:3}
\end{equation}
with $C=4e^{-2}$. To see the last relation, use the elementary bound $x^2e^{-x}\leq 4e^{-2}$ for $x\geq0$:
\[
\int_0^\infty e^{-tx}\dd\mu(x)=\frac1{t^2}\int_0^\infty (tx)^2e^{-tx}\frac{\dd\mu(x)}{x^2}\leq\frac{C}{t^2}
\int_0^\infty\frac{\dd\mu(x)}{x^2}.
\]

\subsection{The simplicity of the spectrum}
We recall some key definitions. Let $A$ be a self-adjoint operator in a Hilbert space $\calH$. A vector $w\in\calH$ is said to be \emph{cyclic} for $A$, if the set of elements 
 \[
 \{\varphi(A)w: \text{ $\varphi$ is bounded continuous} \}
 \]
 is dense in $\calH$.  One says that \emph{the spectrum of $A$ is simple} if there exists a cyclic element for $A$. If $w$ is a cyclic element for $A$, then any element of the form $g(A)w$, where $g$ is a continuous non-vanishing function, is also cyclic. 

Our first preliminary result is the following statement. 
\begin{theorem}\label{prp.z1}
Assume that $\mu$ is either finite or co-finite. Then the spectrum of the restriction $\Gamma_\mu^\perp$ is simple.
\end{theorem}
If $\mu$ is a Carleson measure, this statement can be derived from the results of \cite{MPT}. 
In any case, we give a proof in Section~\ref{sec.a}. 

Note that for a general $\mu\in\calM$, the spectrum of $\Gamma_\mu^\perp$ does not have to be simple; for example, the spectrum of the Carleman operator has multiplicity two.

\subsection{The distinguished cyclic elements}\label{sec.z6}
Our next step is to identify a distinguished cyclic element of $\Gamma_\mu^\perp$ if $\mu$ is either finite or co-finite. 
\emph{Formally,} these cyclic elements can be defined by 
\begin{align}
v_\mu&:=\Gamma_\mu^{1/2}\delta_0, \text{ if $\mu\in\calM_\finite$,}
\label{eq:z9b}
\\
w_\mu&:=\Gamma_\mu^{1/2}\1, \text{ if $\mu\in\calM_\cofinite$,}
\label{eq:z9bb}
\end{align}
where $\delta_0$ is the delta-function (the point mass at the origin) and $\1$ is the function identically equal to $\1$. In order to give a rigorous definition of $v_\mu$ and $w_\mu$, we need to rewrite the formal definitions \eqref{eq:z9b} and \eqref{eq:z9bb} in a weak form. 

\textbf{Notation:} Below $\jap{f,g}$ is the inner product of $f,g\in L^2$, linear in $f$ and anti-linear in $g$, and $\norm{f}$ is the norm of $f\in L^2$.

Let $g\in C^\infty_{\comp}(\bbR_+)$; from \eqref{eq:z9b} and \eqref{eq:z9bb} we \emph{formally} derive 
\begin{align*}
\jap{v_\mu,\Gamma_\mu^{1/2}g}
=
\jap{\Gamma_\mu^{1/2}v_\mu,g}
=
\jap{\Gamma_\mu\delta_0,g}
=
\jap{h_\mu,g}
=
\int_0^\infty h_\mu(t)\overline{g(t)}\dd t, 
\end{align*}
and similarly 
\begin{align*}
\jap{w_\mu,\Gamma_\mu^{1/2}g}
=
\jap{\Gamma_\mu^{1/2}w_\mu,g}
=
\jap{\Gamma_\mu\1,g}
=
\int_0^\infty\int_0^\infty h_\mu(t+s)\overline{g(t)}\dd t\, \dd s.
\end{align*}
Taking into account \eqref{eq:2} and \eqref{eq:3}, we see that the integrals in the right hand sides here converge absolutely. This motivates the following statement. 

\begin{theorem}\label{prp.z2}
\begin{enumerate}[\rm (i)]
\item
Let $\mu\in\calM_\finite$. There exists a unique element $v_\mu\in\overline{\Ran\Gamma_\mu}$ such that for any function $g\in C^\infty_\comp(\bbR_+)$ we have 
\begin{equation}
\jap{v_\mu,\Gamma_\mu^{1/2}g}
=\int_0^\infty h_\mu(t)\overline{g(t)}\dd t.
\label{z12}
\end{equation}
Moreover, this element is cyclic for $\Gamma_\mu^\perp$ and we have 
\begin{equation}
\norm{v_\mu}^2=\mu(\bbR_+).
\label{z9}
\end{equation}
\item
Let $\mu\in\calM_\cofinite$. There exists a unique element $w_\mu\in\overline{\Ran\Gamma_\mu}$ such that for any function $g\in C^\infty_\comp(\bbR_+)$ we have 
\begin{equation}
\jap{w_\mu,\Gamma_\mu^{1/2}g}
=\int_0^\infty\int_0^\infty h_\mu(t+s)\overline{g(t)}\dd t\, \dd s.
\label{z12a}
\end{equation}
Moreover, this element is cyclic for $\Gamma_\mu^\perp$ and we have 
\begin{equation}
\norm{w_\mu}^2=\int_0^\infty \frac{\dd\mu(x)}{x^2}. 
\label{z9a}
\end{equation}
\end{enumerate}
\end{theorem}

We recall (see Proposition~\ref{thm.z1}) that $C^\infty_\comp(\bbR_+)\subset\Dom Q_\mu=\Dom\Gamma_\mu^{1/2}$ and so the left hand sides in \eqref{z12} and \eqref{z12a} are well-defined. 

Using the formal definitions \eqref{eq:z9b} and \eqref{eq:z9bb}, let us give some heuristic explanation of the relations \eqref{z9} and \eqref{z9a}. We have
\[
\norm{v_\mu}^2
=\jap{\Gamma_\mu\delta_0,\delta_0}
=\jap{h_\mu,\delta_0}
=h_\mu(0)=\mu(\bbR_+)
\]
and similarly 
\begin{align*}
\norm{w_\mu}^2
&=\jap{\Gamma_\mu\1,\1}
=\int_0^\infty \int_0^\infty h_\mu(t+s)\dd t\, \dd s
\\
&=\int_0^\infty \int_0^\infty \int_0^\infty e^{-(t+s)x}\dd\mu(x)\, \dd t\, \dd s
=\int_0^\infty \frac{\dd\mu(x)}{x^2}. 
\end{align*}

\subsection{The spectral measures $\sigma$ and $\rho$}
Let $A$ be a self-adjoint operator in a Hilbert space $\calH$. We recall that a (scalar) \emph{spectral measure} $\sigma=\sigma_v$ of $A$ with respect to the element $v$ (whether $v$ is cyclic or not) is uniquely defined by the relation
\begin{equation}
\jap{\varphi(A)v,v}_{\calH}=\int_{\bbR}\varphi(x)\dd\sigma(x), 
\label{eq:defrho}
\end{equation}
where $\varphi$ is any bounded continuous test function, and $\jap{\bigcdot,\bigcdot}_{\calH}$ is the inner product in $\calH$. 

\begin{definition*}[Spectral measures $\sigma$ and $\rho$]
\begin{enumerate}[\rm (i)]
\item
Let $\mu\in\calM_\finite$ and let $v_\mu$ be as described in Theorem~\ref{prp.z2}(i). We define $\sigma$ to be the spectral measure of $\Gamma_\mu$ with respect to the element $v_\mu$: 
\begin{equation}
\jap{\varphi(\Gamma_\mu)v_\mu,v_\mu}
=
\int_0^\infty \varphi(\lambda)\dd\sigma(\lambda),
\label{z10}
\end{equation}
for all bounded continuous functions $\varphi$. 
\item
Let $\mu\in\calM_\cofinite$ and let $w_\mu$ be as described in Theorem~\ref{prp.z2}(ii). We define $\rho$ to be the spectral measure of $\Gamma_\mu$ with respect to the element $w_\mu$: 
\begin{equation}
\jap{\varphi(\Gamma_\mu)w_\mu,w_\mu}
=
\int_0^\infty \varphi(\lambda)\dd\rho(\lambda),
\label{z10a}
\end{equation}
for all bounded continuous functions $\varphi$. 
\end{enumerate}
\end{definition*}
Using the formal definitions \eqref{eq:z9b} and \eqref{eq:z9bb}, we can rewrite \eqref{z10} and \eqref{z10a} in a non-rigorous but illuminating way:
\begin{align}
\jap{\varphi(\Gamma_\mu)\Gamma_\mu\delta_0,\delta_0}
&=
\int_0^\infty \varphi(\lambda)\dd\sigma(\lambda),
\label{z11}
\\
\jap{\varphi(\Gamma_\mu)\Gamma_\mu\1,\1}
&=
\int_0^\infty \varphi(\lambda)\dd\rho(\lambda).
\label{z11a}
\end{align}
In other words, $\dd\sigma(\lambda)/\lambda$ can be understood as the spectral measure of $\Gamma_\mu^\perp$ corresponding to $\delta_0$, and $\dd\rho(\lambda)/\lambda$ as the spectral measure corresponding to $\1$. 

By definition, $\sigma$ and $\rho$ are finite and supported on $\bbR_+$. Since $v_\mu$ and $w_\mu$ are orthogonal to $\Ker\Gamma_\mu$ (see Theorem~\ref{prp.z2}), the measures $\sigma$ and $\rho$ do not have a point mass at the origin. Thus, both $\sigma$ and $\rho$ are in $\calM_\finite$.

\section{Main results}

\subsection{The maps $\Omega$ and $\Omega^{\#}$}\label{sec.z4}
We thus have (nonlinear) \emph{spectral maps} 
\begin{align*}
\Omega:\calM_\finite\to\calM_\finite, \quad  \mu\mapsto \sigma,
\\
\Omega^{\#}:\calM_\cofinite\to\calM_\finite, \quad \mu\mapsto \rho.
\end{align*}
Our first main result is
\begin{theorem}\label{thm.z2}
The spectral map $\Omega$ is an involution on $\calM_\finite$: 
\[
\Omega\circ\Omega=\text{id}.
\] 
In particular, it is a bijection on $\calM_\finite$. Moreover, $\Omega$ preserves the total mass of measures:
\[
\sigma(\bbR_+)=\mu(\bbR_+), \quad\text{ if }\quad \sigma=\Omega(\mu).
\]
\end{theorem}

We emphasize that this includes three non-trivial statements:
\begin{itemize}
\item
$\Omega$ is an injection, i.e. $\Gamma_\mu$ is uniquely determined by $\sigma$.
\item
$\Omega$ is a surjection, i.e. every measure $\sigma\in\calM_\finite$ is the spectral measure of some Hankel operator $\Gamma_\mu$.
\item
$\mu$ can reconstructed from $\sigma$ by computing the spectral measure of $\Gamma_\sigma$.
\end{itemize}

For our second main result, we will need the involution $\#$ on measures induced by the change of variable $x\mapsto 1/x$ on $\bbR_+$. If $\mu$ is any locally finite positive measure on $\bbR_+$, then the measure $\mu^{\#}$ is defined by 
\[
\int_0^\infty \varphi(x)\dd\mu^{\#}(x)
=
\int_0^\infty \varphi\left(\frac1y\right)\frac{\dd\mu(y)}{y^2}
\]
for bounded continuous test functions $\varphi$ compactly supported on $\bbR_+$. For example, if $\dd\mu(x)=w(x)\dd x$, then $\dd\mu^{\#}(x)=w(1/x)\dd x$. 
It is clear that the involution $\#$ is a bijection between $\calM_\finite$ and $\calM_\cofinite$. 

\begin{theorem}\label{thm.z2a}
We have 
\[
\Omega^{\#}=\Omega\circ \#
\]
on $\calM_\cofinite$. In particular, $\Omega^{\#}$ is a bijection between $\calM_\cofinite$ and $\calM_\finite$ and the inverse is given by 
\[
(\Omega^{\#})^{-1}=\#\circ\Omega.
\]
\end{theorem}

\subsection{Corollaries}
Coming back to the class of \emph{bounded} Hankel operators, we obtain the following corollary of Theorem~\ref{thm.z2}. Denote by $\calM_{\text{Carleson}}\subset\calM_\finite$ the subset of finite measures satisfying the Carleson condition \eqref{z3}, and let $\calM_{\text{bounded}}\subset\calM_\finite$ be the subset of finite measures with a bounded support. 

\begin{corollary}\label{cr.z5}
The spectral map $\Omega$ is an involution on $\calM_{\mathrm{Carleson}}\cap\calM_{\mathrm{bounded}}$. 
\end{corollary}
\begin{proof}
As already mentioned, the Carleson condition on $\mu$ is equivalent to the boundedness of $\Gamma_\mu$, i.e. the boundedness of $\Gamma_\mu^\perp$. Furthermore, since $v_\mu$ is a cyclic element for $\Gamma_\mu^\perp$, the boundedness of $\Gamma_\mu^\perp$ is equivalent to the boundedness of the support of the spectral measure $\sigma$. Thus, we have 
\[
\mu\in\calM_{\mathrm{Carleson}}\Leftrightarrow \Omega(\mu)\in \calM_{\mathrm{bounded}}.
\]
Since $\Omega$ is an involution, we also get
\[
\mu\in\calM_{\mathrm{bounded}}\Leftrightarrow \Omega(\mu)\in \calM_{\mathrm{Carleson}}.
\]
It follows that $\Omega$ is a bijection on the set $\calM_{\mathrm{Carleson}}\cap\calM_{\mathrm{bounded}}$, as required.  
\end{proof}
If the reader wants to avoid the technicalities associated with unbounded operators $\Gamma_\mu$ in our construction below, they are invited to keep in mind the class of measures $\calM_{\mathrm{Carleson}}\cap\calM_{\mathrm{bounded}}$ and simply follow the algebra in the proofs.  

There is another interesting special case. Let $\calM_{0}\subset\calM_\finite$ be the set of measures that are finite linear combinations of point masses. As already mentioned, if $\mu\in \calM_{0}$, then $\Gamma_\mu$ is a finite rank operator and so $\sigma\in\calM_{0}$. 
\begin{corollary}\label{cr.z4}
The spectral map $\Omega$ is an involution on $\calM_{0}$.
\end{corollary}
\begin{proof}
It is straightforward to see that  $\mu\in\calM_{0}$ implies that $\Gamma_\mu$ is finite rank, and so $\sigma\in\calM_{0}$. Thus $\Omega(\calM_{0})\subset \calM_{0}$. Since $\Omega$ is an involution on $\calM_\finite$, we obtain that $\Omega(\calM_{0})=\calM_{0}$ and $\Omega$ is an involution on $\calM_{0}$.
\end{proof}

The fact that $\Omega$ is a bijection on $\calM_{0}$ is a consequence of the earlier work by O.~Pocovnicu \cite{Pocovnicu}, who considered a more general class of finite rank Hankel operators (not necessarily self-adjoint). She worked in the Hardy space of the upper half-plane rather than in $L^2(\bbR_+)$, and so her results are expressed in a very different language. See also subsequent work \cite{GP}. 

\subsection{The case $\mu\in \calM_\finite\cap\calM_\cofinite$}
It is interesting to consider the special case when the measure $\mu$ is both finite and co-finite.
Then both $v_\mu$ and $w_\mu$ are cyclic elements for $\Gamma_\mu$, and thus both measures $\sigma=\Omega(\mu)$ and $\rho=\Omega^{\#}(\mu)$ are well defined. What is the relationship between $\sigma$ and $\rho$? 

From Theorem~\ref{thm.z2} we find that $\mu=\Omega(\sigma)$ and from Theorem~\ref{thm.z2a} we have
\[
\rho=\Omega\circ\#(\mu)=\Omega\circ\#\circ \Omega(\sigma).
\]
This is the relationship between $\sigma$ and $\rho$. 

Note that the naive guess that $\Omega\circ\#\circ \Omega(\sigma)=\sigma^\#$ is \textbf{wrong}. Indeed, both $\sigma$ and $\rho$ are supported on the same set, i.e. the spectrum of $\Gamma_\mu$. 
If we had $\rho=\sigma^{\#}$, then $\rho$ and $\sigma$ would have been supported on different sets.

\subsection{Some properties of the spectral map $\Omega$}
We have already seen that $\Omega$ preserves the total mass of measures. Let us discuss other simple properties of $\Omega$. To give some intuition, we note that $\Omega$ acts on measures supported at a single point as follows. Let $\delta_a$ be the probability measure which is a single point mass at $x=a$. Then 
\[
\Omega: A\delta_a\mapsto A\delta_{A/2a}, \quad a>0,\quad A>0.
\]
Returning to general measures, let us list the behavior of $\Omega$ with respect to the scaling of $\mu$. For $\mu\in\calM_\finite$ and $\tau>0$, we denote by $\mu_\tau$ the scaled measure defined by 
\[
\int_0^\infty \varphi(x)\dd\mu_\tau(x)=\int_0^\infty \varphi(\tau x)\dd\mu(x)
\]
for bounded continuous test functions $\varphi$. For example, if $\dd\mu(x)=w(x)\dd x$, then $\dd\mu_\tau(x)=\frac1{\tau}w(x/\tau)\dd x$. 
\begin{theorem}\label{thm.z6}
For $\mu\in\calM_{\finite}$, let $\sigma=\Omega(\mu)$ and let $\tau>0$. Then
\[
\Omega(\tau\mu)=\tau\sigma_\tau\quad\text{ and }\quad\Omega(\mu_\tau)=\sigma_{1/\tau}.
\]
\end{theorem}
Next, we give two identities which are essentially trace formulas for $\Gamma_\mu$. 
\begin{theorem}\label{thm.z7}
For $\mu\in\calM_{\finite}$, let $\sigma=\Omega(\mu)$. The following statements are equivalent:
\begin{enumerate}[\rm (i)]
\item
$\Gamma_\mu$ is trace class;
\item
one has 
\[
\int_0^\infty \frac{\dd\mu(x)}{x}<\infty;
\]
\item
$\sigma$ is a purely atomic measure supported on a summable sequence of points $\{a_k\}_{k=1}^\infty$.
\end{enumerate}
If conditions {\rm(i)--(iii)} hold true, then we have the trace formula
\begin{equation}
\Tr\Gamma_\mu=\frac12\int_0^\infty\frac{\dd\mu(x)}{x}=\sum_{k=1}^\infty a_k.
\label{z15}
\end{equation}
\end{theorem}

\begin{theorem}\label{thm.z8}
For $\mu\in\calM_{\finite}$, let $\sigma=\Omega(\mu)$. The following statements are equivalent:
\begin{enumerate}[\rm (i)]
\item
$\Gamma_\mu$ is Hilbert-Schmidt;
\item
one has 
\[
\int_0^\infty \int_0^\infty \frac{\dd\mu(x)\dd\mu(y)}{(x+y)^2}<\infty;
\]
\item
$\sigma$ is a purely atomic measure supported on a square-summable sequence of points $\{a_k\}_{k=1}^\infty$.
\end{enumerate}
If conditions {\rm(i)--(iii)} hold true, then we have the trace formula
\begin{equation}
\Tr\Gamma_\mu^2
=\int_0^\infty \int_0^\infty \frac{\dd\mu(x)\dd\mu(y)}{(x+y)^2}
=\sum_{k=1}^\infty a_k^2.
\label{z17}
\end{equation}
\end{theorem}

\subsection{The continuity of $\Omega$}
Finally, we prove the continuity of $\Omega$ with respect to the weak convergence of measures. Let $\{\mu_n\}_{n=1}^\infty$ and $\mu$ be measures in $\calM_\finite$; we recall that (see e.g. \cite[Section 30]{Bauer}) $\mu_n\to\mu$ weakly as $n\to\infty$ means that for any function $f$ bounded and continuous on $[0,\infty)$ we have 
\[
\int_0^\infty f(x)\dd\mu_n(x)\to \int_0^\infty f(x)\dd\mu(x), \quad n\to\infty.
\]
In particular, taking $f=1$ this means that $\mu_n(\bbR_+)\to\mu(\bbR_+)$. 
(Weak convergence is really the weak-$*$ convergence with measures considered as linear functionals on  $C\ti{b}([0,\infty))$.)
\begin{theorem}\label{thm.z9}
If $\mu_n\to\mu$ weakly, then $\Omega(\mu_n)\to\Omega(\mu)$ weakly.  
\end{theorem}
The proof is given in Section~\ref{sec.ct}. We note that the topology of weak convergence of finite measures on $\bbR$ is metrizable (see e.g. \cite{Bauer}, Exercise 4 on page 216), and so Theorem~\ref{thm.z9} ensures
the continuity of $\Omega$ in this topology. 

\section{Discussion and key ideas of the proof}\label{sec.aa}
\subsection{Notation}
Let $\mu\in\calM$.  
Along with $L^2=L^2(\bbR_+)$, we also consider the space $L^2(\mu):=L^2(\bbR,\dd\mu)$. We denote the inner product in  $L^2(\mu)$ by $\jap{f,g}_\mu$ and the norm by $\norm{f}_\mu$. 
Our inner products are linear in the first argument and anti-linear in the second one. We denote by $\1_\mu$ the function in $L^2(\mu)$ that is identically equal to $1$. 

\subsection{Reduction to $L^2(\mu)$}\label{sec.aa2}
Let us informally describe the main ideas of the proof. They originate in control theory, but here we present them in more traditional operator theoretic language. (For a discussion of the control theory aspects see Appendix.) We recall a general fact: if $L$ is a closed operator acting between two Hilbert spaces, then the operators
\begin{equation}
L^*L|_{(\Ker L)^\perp}\quad \text{ and }\quad LL^*|_{(\Ker L^*)^\perp}
\label{z6}
\end{equation}
are unitarily equivalent. The first idea is to represent the Hankel operator $\Gamma_\mu$ as a product of this kind and swap the factors.

Let $L:L^2\to L^2(\mu)$ be the operator of the Laplace transform:
\[
(Lf)(x)=\int_0^\infty e^{-tx}f(t)\dd t, \quad x>0,
\]
and let $L^*: L^2(\mu)\to L^2$ be its adjoint. For a general $\mu\in\calM$, the operators $L$ and $L^*$ do not have to be bounded. We postpone the (non-trivial) discussion of the domains and the precise definition of $L$ and $L^*$ to Section~\ref{sec.b}. For the moment, we brush the technical details aside and explain the formal structure of the proof. 
Since 
\[
h_\mu(t+s)=\int_0^\infty e^{-tx}e^{-sx}\dd\mu(x),
\]
the operator $\Gamma_\mu$ can be represented in a factorised form
\[
\Gamma_\mu=L^*L.
\]
Now consider the \emph{model operator}
\[
G_\mu:=LL^*\text{ in $L^2(\mu)$.}
\]
Formally, this is the integral operator
\[
(G_\mu f)(x)=\int_0^\infty\frac{f(y)}{x+y}\dd\mu(y).
\]
Below we will see that
\begin{itemize}
\item
For all $\mu\in\calM$, we have $\Ker L^*=\Ker G_\mu=\{0\}$ and therefore, by \eqref{z6}, the operators 
\[
\Gamma_\mu^\perp \quad\text{ and }\quad {G_\mu}
\]
are unitarily equivalent.
\item
If $\mu\in\calM_\finite$, then the spectrum of ${G_\mu}$ is simple and the element $\1_\mu$ ($=$ the function in $L^2(\mu)$ that is identically equal to $1$) is a cyclic element for ${G_\mu}$. The spectral measure of $G_\mu$ corresponding to $\1_\mu$ coincides with $\sigma$. 
\item
if $\mu\in\calM_\cofinite$, then the spectrum of ${G_\mu}$ is simple and the function 
\[
\theta_\mu(x)=1/x
\] 
is a cyclic element for ${G_\mu}$. The spectral measure of $G_\mu$ corresponding to $\theta_\mu$ coincides with $\rho$. 
\end{itemize}
Thus, the problem reduces to the spectral analysis of the operator ${G_\mu}$ in $L^2(\mu)$ with the cyclic element $\1_\mu$ if $\mu$ is finite and $\theta_\mu$ if $\mu$ is co-finite.

\subsection{Finite $\mu$: the Lyapunov equation}\label{sec.aa3}
Consider the case of $\mu\in\calM_\finite$. 
Along with ${G_\mu}$, let us consider the operator $X_\mu$ of multiplication by the independent variable in $L^2(\mu)$: 
\[
(X_\mu f)(x)=xf(x).
\]
Observe that
\begin{align}
(X_\mu{G_\mu} f)(x)+({G_\mu} X_\mu f)(x)
&=
\int_0^\infty\frac{xf(y)}{x+y}\dd\mu(y)
+
\int_0^\infty\frac{yf(y)}{x+y}\dd\mu(y)
\notag
\\
&=
\int_0^\infty \frac{x+y}{x+y}f(y)\dd\mu(y)
=
\int_0^\infty f(y)\dd\mu(y).
\label{z8}
\end{align}
Thus, we have the \emph{Lyapunov equation}
\begin{equation}
X_\mu{G_\mu}+{G_\mu} X_\mu=\jap{\bigcdot,\1_\mu}_\mu \1_\mu.
\label{c2}
\end{equation}
The important feature of the Lyapunov equation is that ${G_\mu}$ and $X_\mu$ enter this equation in a symmetric fashion. Let us summarise:
\begin{itemize}
\item
$\1_\mu$ is a cyclic element for $X_\mu$, with the corresponding spectral measure $\mu$;
\item
$\1_\mu$ is a cyclic element for ${G_\mu}$, with the corresponding spectral measure $\sigma$.
\end{itemize}
From here it is not hard to see that  $\mu\mapsto\sigma$ and $\sigma\mapsto\mu$ is one and the same map on measures. In other words, using this equation we will show that the map $\Omega$ is an involution. 

\subsection{Co-finite $\mu$: change of variable}
In terms of the operator $G_\mu$ in $L^2(\mu)$, the case of a co-finite $\mu$ reduces to the case of a finite $\mu$ by a simple change of variable $x\mapsto 1/x$ on $\bbR_+$. Indeed, let $\mu\in\calM_\cofinite$; consider the map 
\[
\Theta:L^2(\mu)\to L^2(\mu^{\#}), \quad (\Theta f)(x)=f(1/x)/x.
\]
A trivial calculation with the change of variable $x\mapsto 1/x$ shows that $\Theta$ is unitary and 
\[
\Theta^*G_{\mu^{\#}}\Theta=G_\mu, \quad \Theta \theta_\mu=\1_{\mu^{\#}}.
\]
This reduces the problem for the spectral measure of $G_\mu$, $\mu\in\calM_\cofinite$, to the problem for the spectral measure of $G_{\mu^{\#}}$, $\mu\in\calM_\finite$. Theorem~\ref{thm.z2a} follows from here easily.

\begin{remark*}
Formally multiplying the Lyapunov equation by $X_\mu^{-1}$ on both sides, we obtain the dual equation
\[
X_\mu^{-1}G_\mu+G_\mu X_\mu^{-1}=\jap{\bigcdot,\theta_\mu}\theta_\mu.
\]
We could treat the case of co-finite measures $\mu$ by a direct analysis of this equation, exactly mirroring the case of finite $\mu$. However, it is much easier to use the change of variable $x\mapsto 1/x$ as described above. 
\end{remark*}

\subsection{Analogy with Jacobi matrices and Schr\"odinger operators}\label{sec.jacobi}
Let us take a wider look at the type of problems discussed in this paper. Let $A$ be a self-adjoint operator in a Hilbert space with a cyclic element $v$, and let $\sigma$ be the corresponding spectral measure. For some \emph{very special} operators $A$ and \emph{very special} choices of $v$, the spectral map 
\[
\Omega: A\mapsto \sigma
\]
may happen to be injective. In yet more special cases, the image of $\Omega$ may be explicitly described, i.e. $\Omega$ can be realized as a bijection. Theorem~\ref{thm.z2} is a new example of this type, with an additional extra special property that $\Omega$ can be viewed as an involution. To provide context, we mention two other classical examples. 

Consider Jacobi matrices
\[
J=\begin{pmatrix}
b_0 & a_0 & 0 & 0 & 0 &\cdots\\
a_0 & b_1 & a_1 & 0 & 0 & \cdots\\
0 & a_1 & b_2 & a_2 & 0 &\cdots\\
0 & 0 & a_2 & b_3 & a_3 & \cdots\\
\vdots&\vdots&\vdots&\vdots&\vdots&\ddots
\end{pmatrix}, \quad b_n\in\bbR, \quad a_n\in\bbR_+
\quad \text{ in $\ell^2$.}
\]
For simplicity of discussion let us assume that $J$ is bounded.
It is an elementary fact that $v=(1,0,0,\dots)\in \ell^2$ is a cyclic element for $J$. 
Moreover, $J$ is uniquely determined by the spectral measure $\sigma$ corresponding to $v$. Further, the image of the spectral map is the set of all probability measures $\sigma$ with bounded infinite support. Thus, in this case the spectral map $J\mapsto \sigma$ can be viewed as a bijection between bounded Jacobi matrices and probability measures with bounded infinite support. See e.g. \cite{A} for a detailed exposition of these facts and the background. 

Another classical example is the self-adjoint Schr\"odinger operator 
\[
S_q=-\frac{d^2}{dx^2}+q(x) \quad \text{ in $L^2(\bbR_+)$}
\] 
with a real-valued potential $q$. For definiteness, let us assume that $q\in L^\infty(\bbR_+)$ and that $S_q$ is supplied with the Neumann boundary condition at the origin $x=0$. Then the spectral measure $\sigma$ of $S_q$ may be defined as in \eqref{eq:defrho} with $v=\delta_0$, where $\delta_0$ is the delta-function at the origin; see e.g. \cite[Section 3]{GS} for the details of this point of view. A more standard definition of the spectral measure $\sigma$ involves the Titchmarsh-Weyl $m$-function. 

Whatever is the definition of $\sigma$, the classical Borg-Marchenko uniqueness theorem asserts that the Schr\"odinger operator $S_q$ is uniquely defined by $\sigma$. Thus, the spectral map $S_q\mapsto \sigma$ is an injection. Of course, the image of the spectral map depends on the choice of the class of the potentials $q$; this description is a subtle problem. 

\subsection{The structure of the paper}
In Section~\ref{sec.b} we discuss the rigorous definition of the Laplace transform $L$ and its adjoint. In Section~\ref{sec.c}, we perform most of the analysis in $L^2(\mu)$. In Section~\ref{sec.a}, we return to $L^2$ and prove Theorems~\ref{thm.z2} and \ref{thm.z2a}.
In Section~\ref{sec.p} we prove the scaling properties of $\Omega$ and the trace formulas. 
In Section~\ref{sec.exa} we consider two examples of Hankel operators where the spectral measures can be computed explicitly.

\section{The Laplace transform}\label{sec.b}

We fix $\mu\in\calM$ and denote by $\calL$, $\calL_\mu$ the Laplace transform, 
\begin{align*}
\calL f(x)&=\int_0^\infty e^{-tx}f(t)\dd t, \quad\quad x>0,\quad\quad f\in L^2,
\\
\calL_\mu f(t)&=\int_0^\infty e^{-tx}f(x)\dd\mu(x), \quad t>0,\quad f\in L^2(\mu).
\end{align*}
Here we view $\calL f$ and $\calL_\mu f$ as functions (pointwise) rather than elements of function spaces. For a general $f\in L^2$, the Laplace transform $\calL f$ is not necessarily in $L^2(\mu)$, and likewise for a general $f\in L^2(\mu)$ the Laplace transform $\calL_\mu f$ is not necessarily in $L^2$. Let $L$ and $L_\mu$ be the linear operators defined as the restrictions of $\calL$, $\calL_\mu$ onto their \emph{maximal domains}:
\begin{align*}
L f&=\calL f, & f\in \Dom L:&=\{f\in L^2: \calL f\in L^2(\mu)\},
\\
L_\mu f&=\calL_\mu f,  & f\in \Dom L_\mu:&=\{f\in L^2(\mu): \calL_\mu f\in L^2\}.
\end{align*}
Let $\calD_\mu\subset L^2(\mu)$ be the set of bounded compactly supported functions on $\bbR_+$; this means, in particular, that the supports are at a positive distance from the origin.

\begin{theorem}\label{thm.b1}\cite{Ya1,PuTreil}
Let $\mu\in\calM$. The operators  $L$ and $L_\mu$ are closed. We have 
\[
C^\infty_{\comp}(\bbR_+)\subset\Dom L, \quad \calD_\mu\subset\Dom L_\mu.
\]
Moreover, $C^\infty_{\comp}(\bbR_+)$ is dense in $\Dom L$ with respect to the graph norm of $L$ and $\calD_\mu$ is dense in $\Dom L_\mu$ with respect to the graph norm of $L_\mu$.  We have 
\[
L_\mu^*=L \quad\text{ and }\quad L^*=L_\mu.
\]
The kernel of $L_\mu$ is trivial. The kernel of $L$ is either trivial or infinite-dimensional. It is infinite-dimensional if and only if $\mu$ is a purely atomic measure supported on a sequence of points $\{a_k\}_{k\in\bbZ_+}\subset \bbR_+$ satisfying the Blaschke condition \eqref{z0a}. 
\end{theorem}

\textbf{Discussion:}
\begin{itemize}
\item
The most difficult part of the theorem is the density of $C^\infty_{\comp}(\bbR_+)$ (resp. $\calD_\mu$) in $\Dom L$ (resp. $\Dom L_\mu$).  
\item
In Section 3 of \cite{Ya1}, this theorem (except for the last part) was proved under the additional assumption \eqref{b0a}. This assumption was lifted in \cite{PuTreil}. 
\item
Returning to the quadratic form $Q_\mu$ of $\Gamma_\mu$, we see that 
\[
Q_\mu[f]=\norm{Lf}^2,
\]
$\Dom Q_\mu=\Dom L$ and $\Ker \Gamma_\mu=\Ker L$. Thus, Propositions~\ref{thm.z1} and \ref{prp.z1d} are consequences of Theorem~\ref{thm.b1}. 
\end{itemize}

\section{Analysis in the space $L^2(\mu)$}\label{sec.c}
\subsection{The operator $G_\mu$}
Let $\mu\in\calM$. 
Here we consider the operator ${G_\mu}$ in $L^2(\mu)$, defined by
\[
{G_\mu}=L_\mu^*{L_\mu}
\]
with the domain
\[
\Dom {G_\mu}=\{f\in \Dom {L_\mu}: {L_\mu} f\in \Dom L_\mu^*\}.
\]
With this definition the quadratic form of $G_\mu$ can be written as
\begin{equation}
\norm{G_\mu^{1/2}f}_\mu^2=\int_0^\infty\abs{\calL_{\mu}f(t)}^2\dd t, \quad f\in \Dom G_\mu^{1/2}=\Dom L_\mu.
\label{c11}
\end{equation}
We would like to be able to write 
\begin{equation}
({G_\mu} f)(x)=\int_0^\infty \frac{f(y)}{x+y}\dd\mu(y), \quad x>0
\label{c1}
\end{equation}
at least for sufficiently ``nice'' functions $f$. 

\begin{lemma}\label{lma.c1}
Suppose $\mu\in\calM$ is either finite or co-finite. Then $\calD_\mu\subset\Dom {G_\mu}$ and the operator $G_\mu$ acts according to \eqref{c1} on elements $f\in\calD_\mu$. 
\end{lemma}
\begin{proof}
Let $f\in \calD_\mu$, with $\supp f\subset[a,\infty)$ for some $a>0$. Then for $g=\mathcal {L_\mu} f$ we have 
\[
\abs{g(t)}\leq \int_a^\infty e^{-xt}\abs{f(x)}\dd\mu(x)
\leq C e^{-at},
\]
and therefore 
\[
\abs{\mathcal L g(x)}\leq C\int_0^\infty e^{-at}e^{-xt}\dd t=
C/(a+x).
\]
We note that in both cases $\mu\in\calM_\finite$ and $\mu\in\calM_\cofinite$, the measure $\mu$ satisfies the condition
\[
\int_0^\infty \frac{\dd\mu(x)}{1+x^2}<\infty,
\]
and therefore $\mathcal Lg\in L^2(\mu)$.  Thus, $g={L_\mu} f\in \Dom L_\mu^*$, or equivalently $f\in \Dom {G_\mu}$. 

Now for $f\in \calD_\mu$, using Fubini,  we find
\[
({G_\mu} f)(x)=(L_\mu^*{L_\mu} f)(x)=\int_0^\infty e^{-tx}\int_0^\infty e^{-ty}f(y)\dd\mu(y) \dd t
=
\int_0^\infty \frac{f(y)}{x+y}\dd\mu(y)
\]
as required. 
\end{proof}
In Sections~\ref{sec.c2}--\ref{sec.c6} below we consider finite $\mu$. In Section~\ref{sec.c7} we return to co-finite $\mu$. 

\subsection{The Lyapunov equation}\label{sec.c2}
Assume that $\mu\in\calM_\finite$. We consider the function $\1_\mu\in L^2(\mu)$ which is identically equal to $1$. Along with $G_\mu$, we consider the operator $X_\mu$ of multiplication by the independent variable in $L^2(\mu)$. More precisely, 
\[
(X_\mu f)(x)=xf(x), \quad 
\Dom X_\mu=\left\{f\in L^2(\mu): \int_0^\infty x^2\abs{f(x)}^2\dd\mu(x)<\infty\right\}.
\]
It is clear that with this definition, $X_\mu$ is self-adjoint. As explained in Section~\ref{sec.aa3}, the importance of $X_\mu$ for us is that it enters the Lyapunov equation \eqref{c2}. The precise interpretation of this equation is a non-trivial question, because both $X_\mu$ and $G_\mu$ are in general unbounded. We use the approach that is standard in control theory: converting the Lyapunov equation into the integral form. Observe that \emph{formally}, 
\[
\frac{d}{dt}e^{-tX_\mu}{G_\mu} e^{-tX_\mu}=-e^{-tX_\mu}({G_\mu} X_\mu+X_\mu{G_\mu}) e^{-tX_\mu}=-\jap{\bigcdot,e^{-tX_\mu}\1_\mu}e^{-tX_\mu}\1_\mu.
\]
Integrating over $t$ from $0$ to infinity, again \emph{formally} we obtain
\[
{G_\mu}=\int_0^\infty \jap{\bigcdot,e^{-tX_\mu}\1_\mu}_\mu e^{-tX_\mu}\1_\mu\dd t,
\]
and similarly 
\[
X_\mu=\int_0^\infty \jap{\bigcdot,e^{-t{G_\mu}}\1_\mu}_\mu e^{-t{G_\mu}}\1_\mu\dd t.
\]
We will justify and use these equations in the weak form, i.e. we will prove
\begin{lemma}\label{lma.c4}
Let $\mu\in\calM_\finite$. Then we have the identities 
\begin{align}
\norm{G_\mu^{1/2}f}_\mu^2&=\int_0^\infty \abs{\jap{f,e^{-tX_\mu}\1_\mu}_\mu}^2\dd t,
\quad \forall f\in\Dom G_\mu^{1/2},
\label{c3a}
\\
\norm{X_\mu^{1/2}f}_\mu^2&=\int_0^\infty \abs{\jap{f,e^{-t{G_\mu}}\1_\mu}_\mu}^2\dd t, 
\quad \forall f\in\Dom G_\mu. 
\label{c4a}
\end{align}
\end{lemma}
We start by observing that \eqref{c3a} is completely straightforward. Indeed, we have 
\[
\jap{f,e^{-tX_\mu}\1_\mu}_\mu=(\calL_\mu f)(t), 
\]
and so \eqref{c3a} is equivalent to \eqref{c11} and holds true for all $f\in\Dom G_\mu^{1/2}$. Moreover, \eqref{c3a} does not require the measure $\mu$ to be finite. If $\mu$ is not finite, then $\1_\mu$ does not make sense as an element of $L^2(\mu)$, but $e^{-tX_\mu}\1_\mu$ still makes sense. 

In contrast, the identity \eqref{c4a} for $X_\mu$ is subtle and \emph{does} require
$\mu$ to be finite (see Remark after Lemma~\ref{lma.c5}). 
We will prove \eqref{c4a} in Section~\ref{sec.aa4}. This requires some technical 
preparations. 

\begin{remark}
One can eventually prove \eqref{c4a} for 
all $f\in\Dom X_\mu^{1/2}$,  but the proof for $f\in\Dom G_\mu$ is slightly easier and this is all we need. 
In fact, Theorem \ref{thm.c7} below gives a unitary equivalence of triples $(G_\mu, X_\mu, \1_\mu)$ and $(X_\sigma, G_\sigma, \1_\sigma)$. Using this and applying \eqref{c3a} to the measure $\sigma$ in place of $\mu$, we obtain \eqref{c4a} for all $f\in \Dom X_\mu^{1/2}$. 
\end{remark}

\subsection{The boundedness of $X_\mu{G_\mu}$ and ${G_\mu} X_\mu$}

\begin{lemma}\label{lma.c2}
Let $\mu\in\calM_\finite$.
\begin{enumerate}[\rm (i)]
\item
If $f\in\Dom X_\mu$, then $X_\mu f\in\Dom{G_\mu}$ and 
\[
\norm{{G_\mu} X_\mu f}_\mu\leq \mu(\bbR_+)\norm{f}_\mu.
\]
 Thus, the operator ${G_\mu} X_\mu$, defined initially on $\Dom X_\mu$, extends to $L^2(\mu)$ as a bounded operator.
\item
If $f\in\Dom{G_\mu}$, then ${G_\mu} f\in\Dom X_\mu$ and 
\[
\norm{X_\mu {G_\mu} f}_\mu\leq \mu(\bbR_+)\norm{f}_\mu.
\] 
Thus, the operator $X_\mu {G_\mu}$, defined initially on $\Dom{G_\mu}$, extends to $L^2(\mu)$ as a bounded operator. 
\end{enumerate}
\end{lemma}
\begin{proof} Part (i):
first let $f\in\calD_\mu$. Then it is clear that $xf\in\calD_\mu$ and 
\[
\abs{({G_\mu} X_\mu f)(x)}
\leq \int_0^\infty\frac{y\abs{f(y)}}{x+y}\dd\mu(y)
\leq \int_0^\infty \abs{f(y)}\dd\mu(y)
\leq 
\mu(\bbR_+)^{1/2}\norm{f}_{\mu},
\]
and therefore 
\[
\norm{{G_\mu} X_\mu f}_\mu
\leq
\mu(\bbR_+)\norm{f}_\mu.
\]
Next, let $f\in\Dom X_\mu$. Choose a sequence $f_n\in\calD_\mu$ such that $f_n\to f$ and $X_\mu f_n\to X_\mu f$ in $L^2(\mu)$. Then by the above estimate the sequence ${G_\mu} X_\mu f_n$ is Cauchy and therefore converges to some element $F\in L^2(\mu)$. Since ${G_\mu}$ is a closed operator, it follows that $X_\mu f\in\Dom{G_\mu}$ and ${G_\mu} X_\mu f=F$. Now passing to the limit, we obtain the bound 
$\norm{{G_\mu} X_\mu f}_\mu\leq\mu(\bbR_+) \norm{f}_\mu$. 

Part (ii): in brief, this follows from part (i) by taking adjoints. In more detail:
let $f\in\Dom{G_\mu}$ and $g\in\Dom X_\mu$. Then using part (i), 
\[
\abs{\jap{X_\mu g,{G_\mu} f}_\mu}
=\abs{\jap{{G_\mu} X_\mu g,f}_\mu}
\leq \norm{{G_\mu} X_\mu g}_\mu\norm{f}_\mu
\leq \mu(\bbR_+)\norm{g}_\mu\norm{f}_\mu.
\]
Since this is true for all $g\in\Dom X_\mu$, we find that ${G_\mu} f\in\Dom X_\mu$ and 
\[
\jap{X_\mu g,{G_\mu} f}_\mu=\jap{g,X_\mu{G_\mu} f}_\mu.
\]
Now $\norm{X_\mu {G_\mu} f}_\mu\leq \mu(\bbR_+)\norm{f}_\mu$ follows from the above estimate. 
\end{proof}

\begin{corollary}
For $\mu\in\calM_\finite$, the Lyapunov equation
\[
X_\mu{G_\mu}+{G_\mu} X_\mu=\jap{\bigcdot,\1_\mu}_\mu \1_\mu
\]
holds true.
\end{corollary}
\begin{proof} 
Perform the calculation \eqref{z8} on $f\in\calD_\mu$ and then take the closure. 
By Lemma~\ref{lma.c2}, both operators on the left hand side are bounded. 
\end{proof}

Next, we need a technical lemma on domains. 

\begin{lemma}\label{lma.c3a}
Let $\mu\in\calM_\finite$.
\begin{enumerate}[\rm (i)]
\item
If $f\in \Dom X_\mu^{1/2}$, then $X_\mu^{1/2}f\in\Dom G_\mu^{1/2}$ and 
\[
\norm{G_\mu^{1/2}X_\mu^{1/2}f}_\mu\leq\mu(\bbR_+)^{1/2}\norm{f}_\mu.
\]
Furthermore, $\Ran X_\mu$ is dense in $\Dom G_\mu^{1/2}$ in the graph norm of $G_\mu^{1/2}$. 
\item
If $f\in \Dom G_\mu^{1/2}$, then $G_\mu^{1/2}f\in\Dom X_\mu^{1/2}$ and 
\[
\norm{X_\mu^{1/2}G_\mu^{1/2}f}_\mu\leq\mu(\bbR_+)^{1/2}\norm{f}_\mu.
\]
Furthermore, $\Ran {G_\mu}$ is dense in $\Dom X_\mu^{1/2}$ in the graph norm of $X_\mu^{1/2}$. 
\end{enumerate}
\end{lemma}
\begin{proof}
(i) Let us rephrase Lemma~\ref{lma.c2}(i) as follows:
$\Dom X_\mu^{-1}\subset\Dom{G_\mu}$ and $\norm{{G_\mu} f}_\mu\leq\mu(\bbR_+) \norm{X_\mu^{-1}f}_\mu$ for all $f\in \Dom X_\mu^{-1}$. Now by the Heinz inequality, it follows that $\Dom X_\mu^{-1/2}\subset\Dom G_\mu^{1/2}$ and $\norm{G_\mu^{1/2} f}_\mu\leq \mu(\bbR_+)^{1/2}\norm{X_\mu^{-1/2}f}_\mu$ for all $f\in \Dom X_\mu^{-1/2}$. This proves the first statement. 

In order to prove the denseness, we observe that $\calD_\mu\subset\Ran X_\mu$. Thus, it suffices to prove that $\calD_\mu$ is dense in $\Dom G_\mu^{1/2}$ in the graph norm of $G_\mu^{1/2}$. But the domain of $G_\mu^{1/2}$ is exactly the form domain of ${G_\mu}$ and the graph norm of $G_\mu^{1/2}$ is exactly the norm induced by the quadratic form \eqref{c11} of $G_\mu$; thus the statement follows from Theorem~\ref{thm.b1}.

(ii) As in the proof of part (i), the first statement follows from Lemma~\ref{lma.c2}(ii) by using the Heinz inequality. Let us prove denseness. Here it is more convenient to use a ``soft'' argument (which we could have used in part (i) as well). Suppose, to get a contradiction, that for some $f\in\Dom X_\mu^{1/2}$ we have 
\begin{equation}
\jap{f,G_\mu g}_\mu+\jap{X_\mu^{1/2}f,X_\mu^{1/2}G_\mu g}_\mu=0, \quad \forall g\in \Dom G_\mu.
\label{c7a}
\end{equation}
By Lemma~\ref{lma.c2}(ii), we have $X_\mu^{1/2}G_\mu g\in \Dom X_\mu^{1/2}$ and so we can rewrite \eqref{c7a} as 
\[
\jap{f,G_\mu g}_\mu+\jap{f,X_\mu G_\mu g}_\mu=0. 
\]
Since $X_\mu G_\mu$ is bounded, from here we get
\[
\abs{\jap{f,G_\mu g}_\mu}=\abs{\jap{f,X_\mu G_\mu g}_\mu}\leq \mu(\bbR_+)\norm{f}_\mu\norm{g}_\mu
\]
and so $f\in\Dom G_\mu$. Thus, we can rewrite \eqref{c7a} as
\[
\jap{G_\mu f,g}_\mu+\jap{G_\mu X_\mu f,g}_\mu=0, \quad \forall g\in\Dom G_\mu.
\]
It follows that $G_\mu(f+X_\mu f)=0$. Since the kernel of $G_\mu$ is trivial, this yields $f+X_\mu f=0$ and therefore $f=0$, as required. 
\end{proof}

\subsection{Proof of Lemma~\ref{lma.c4}}\label{sec.aa4}
As already discussed, we only need to prove \eqref{c4a}. For $f\in\Ran G_\mu$ and $t>0$, we have $e^{-tG_\mu}f\in\Ran G_\mu^{1/2}$ and so, by Lemma~\ref{lma.c3a}(ii), the element $X_\mu^{1/2}e^{-tG_\mu}f$ is well-defined. Thus, we may consider
\[
F(t)=\norm{X_\mu^{1/2}e^{-tG_\mu}f}_\mu^2, \quad t>0.
\]
For $t>0$, we have 
\[
\frac{d}{dt}e^{-tG_\mu}f=-G_\mu e^{-tG_\mu}f.
\]
It follows that for $t>0$, 
\[
F'(t)=
-\jap{X_\mu^{1/2}G_\mu e^{-tG_\mu}f,X_\mu^{1/2}e^{-tG_\mu}f}_\mu
-\jap{X_\mu^{1/2} e^{-tG_\mu}f,X_\mu^{1/2}G_\mu e^{-tG_\mu}f}_\mu.
\]
By Lemma~\ref{lma.c2}(ii), we have $G_\mu e^{-tG_\mu}f\in\Dom{X_\mu}$. Thus, we can rewrite the above expression as
\begin{align*}
F'(t)&=
-\jap{{X_\mu} G_\mu e^{-tG_\mu}f,e^{-tG_\mu}f}_\mu
-\jap{ e^{-tG_\mu}f,{X_\mu} G_\mu e^{-tG_\mu}f}_\mu
\\
&=
-\jap{({X_\mu} G_\mu+G_\mu{X_\mu})e^{-tG_\mu}f, e^{-tG_\mu}f}_\mu
\\
&=-\abs{\jap{e^{-tG_\mu}f,\1_\mu}_\mu}^2,
\end{align*}
where we have used the Lyapunov equation at the last step. 
Integrating from $0$ to $s$, we find
\[
F(0)=F(s)+\int_0^s \abs{\jap{e^{-tG_\mu}f,\1_\mu}_\mu}^2\dd t.
\]
Finally, we wish to take $s\to\infty$. 
Since the kernel of $G_\mu$ is trivial, we have $e^{-sG_\mu}\to0$ in the strong operator topology as $s\to\infty$. Since by assumption $f\in\Ran G_\mu$, we can write $f=G_\mu^{1/2}g$ with some $g\in L^2(\mu)$, which allows us to conclude that 
\[
F(s)
=\norm{X_\mu^{1/2}e^{-sG_\mu}G_\mu^{1/2}g}_\mu^2
=\norm{X_\mu^{1/2}G_\mu^{1/2}e^{-sG_\mu}g}_\mu^2\to0
\]
as $s\to\infty$ (here we have used Lemma~\ref{lma.c3a}(ii)). 
\qed

\subsection{The element $\1_\mu$ is cyclic}

\begin{lemma}\label{lma.c5}
Let $\mu\in\calM_\finite$. Then the spectrum of ${G_\mu}$ is simple and $\1_\mu$ is a cyclic element.
\end{lemma}
\begin{proof}
We use identity \eqref{c4a}. Suppose, to get a contradiction, that $\1_\mu$ is not 
cyclic. Then there exists $f\in\Ran{G_\mu}$ such that $f\perp\varphi({G_\mu})\1_\mu$ for all 
bounded continuous functions $\varphi$. In particular, $f\perp e^{-t{G_\mu}}\1_\mu$ for all $t>0$. 
Then by \eqref{c4a}, we have $X_\mu^{1/2}f=0$. Since the kernel of $X_\mu^{1/2}$ is trivial, it 
follows that $f=0$. 
\end{proof}

\begin{remark*}
We emphasize that the assumption that $\mu$ is finite is essential here. For example, for the Carleman operator (corresponding to $\mu$ being the Lebesgue measure on $\bbR_+$) the spectrum is the interval $[0,\pi]$ with multiplicity \emph{two}, and so there are no cyclic elements. 
\end{remark*}

\subsection{The spectral measure $\sigma$ and the spectral map $\Omega$}\label{sec.c6}
Let $\mu\in\calM_\finite$ and let $\sigma$ be the spectral measure of ${G_\mu}$ associated with the cyclic element $\1_\mu$, i.e. 
\[
\jap{\varphi({G_\mu})\1_\mu,\1_\mu}_\mu
=\int_{\bbR} \varphi(\lambda)\dd\sigma(\lambda)
\]
for bounded continuous functions $\varphi$. 
(In the next section we shall see that this is the same measure as the one defined by \eqref{z10}, \eqref{z11}.)
Since ${G_\mu}$ is positive semi-definite with a trivial kernel, we see that $\sigma$ is supported on $\bbR_+$ with no point mass at $0$; thus $\sigma\in\calM_\finite$. 
We set up the map
\[
\Omega:\calM_\finite\to\calM_\finite, \quad \mu\mapsto \sigma.
\]
(We shall soon see that this is the same map as the one discussed in Section~\ref{sec.z4}.)

\begin{theorem}\label{thm.c7}
For the spectral measure $\sigma$ constructed above
\begin{enumerate}[\rm (i)]
\item 
The triples $(G_\mu, X_\mu, \1_\mu)$ and $(X_\sigma, G_\sigma, \1_\sigma)$ are unitary equivalent, 
i.e.~there exists a unitary operator $U:L^2(\mu)\to L^2(\sigma)$ such that 
\begin{align*}
UG_\mu = X_\sigma U, 
\qquad  U X_\mu = G_\sigma U, 
\qquad U\1_\mu = \1_\sigma. 
\end{align*}
\item 
The map $\Omega$ is an involution on $\calM_\finite$ preserving the total mass of the measure. 
\end{enumerate}
\end{theorem}
\begin{proof}
Part (i): consider the operator $U:L^2(\mu)\to L^2(\sigma)$ defined by 
\[
U:L^2(\mu)\to L^2(\sigma), \quad f({G_\mu})\1_\mu\mapsto f
\]
initially on continuous compactly supported functions $f$. 
Since $\1_\mu$ is a cyclic element for $G_\mu$, it is a standard fact (see e.g. \cite[Lemma 1, Section VII.2]{RS1}) that $U$ extends to a unitary operator such that 
\begin{equation}
U{G_\mu}=X_\sigma U \quad \text{ and }\quad U\1_\mu=\1_\sigma.
\label{c6}
\end{equation}
In particular, $U\Ran{G_\mu}=\Ran X_\sigma$. 

Furthermore, by \eqref{c3a} with $\sigma$ in place of $\mu$, we have
\[
\norm{G_\sigma^{1/2}f}_\sigma^2
=
\int_0^\infty\abs{\jap{f,e^{-tX_\sigma}\1_\sigma}_\sigma}^2\dd t
\]
for any $f\in\Dom G_\sigma^{1/2}$. By Lemma~\ref{lma.c2}(i), we have $\Ran X_\sigma\subset\Dom G_\sigma$, and so in particular the above relation holds true for $f\in\Ran X_\sigma$. Using \eqref{c6}, we rewrite the right-hand side as
\[
\int_0^\infty\abs{\jap{e^{-tX_\sigma}f,U\1_\mu}_\sigma}^2\dd t
=
\int_0^\infty\abs{\jap{U^*e^{-tX_\sigma}f,\1_\mu}_\mu}^2\dd t
=
\int_0^\infty\abs{\jap{e^{-t{G_\mu}}U^*f,\1_\mu}_\mu}^2\dd t.
\]
Observe that for $f\in\Ran X_\sigma$, we have $g:=U^*f\in \Ran G_\mu$. 
Recalling \eqref{c4a}, we obtain
\[
\norm{G_\sigma^{1/2}Ug}_\sigma^2
=
\norm{X_\mu^{1/2}g}_\mu^2, \quad g\in\Ran G_\mu, 
\]
or equivalently 
\begin{equation}
\norm{U^*G_\sigma^{1/2}Ug}_\mu^2
=
\norm{X_\mu^{1/2}g}_\mu^2, \quad g\in\Ran G_\mu.
\label{c9}
\end{equation}
By Lemma~\ref{lma.c3a}(ii), $\Ran {G_\mu}$ is dense in $\Dom X_\mu^{1/2}$ in the graph norm of $X_\mu^{1/2}$. 
Likewise, by Lemma~\ref{lma.c3a}(i) $\Ran X_\sigma$ is dense in $\Dom G_\sigma^{1/2}$ in the graph norm of $G_\sigma^{1/2}$. Taking the closure in \eqref{c9}, we conclude that 
\[
\Dom (U^* G_\sigma^{1/2}U)=\Dom X_\mu^{1/2}
\]
and 
\[
\norm{(U^*G_\sigma^{1/2}U)g}_\mu^2
=
\norm{X_\mu^{1/2}g}_\mu^2
\]
for all $g$ in this common domain. It follows that 

\[
X_\mu=(U^*G_\sigma^{1/2}U)^*(U^*G_\sigma^{1/2}U)=U^* G_\sigma U,
\]
and therefore
\[
G_\sigma=U X_\mu U^*.
\]
Together with \eqref{c6}, this proves (i). 

Part (ii): we need to check that the spectral measure of $G_\sigma$ associated with the element $\1_\sigma$ coincides with the measure $\mu$. For continuous $\varphi$ of compact support we have
\[
\jap{\varphi(G_\sigma)\1_\sigma,\1_\sigma}_\sigma
=
\jap{U\varphi(X_\mu)U^*\1_\sigma,\1_\sigma}_\sigma
=
\jap{\varphi(X_\mu)\1_\mu,\1_\mu}_\mu
=
\int_0^\infty \varphi(\lambda)\dd\mu(\lambda),
\]
as required. Finally, since $\norm{\1_\mu}_\mu^2=\mu(\bbR_+)$, we have $\sigma(\bbR_+)=\mu(\bbR_+)$, and so the total mass of the measure is preserved by the map $\Omega$.
\end{proof}

\subsection{The case $\mu\in\calM_\cofinite$}\label{sec.c7}
Assume $\mu\in\calM_\cofinite$. We denote by $\theta_\mu\in L^2(\mu)$ the function 
\[
\theta_\mu(x)=1/x.
\]
\begin{theorem}\label{thm.c4}
Assume $\mu\in\calM_\cofinite$. 
Then the spectrum of $G_\mu$ is simple and $\theta_\mu$ is a cyclic element.
\end{theorem}
\begin{proof}
Let  $\mu\in\calM_\cofinite$. Consider the map $\Theta$ 
\[
\Theta:L^2(\mu)\to L^2(\mu^{\#}), \quad (\Theta f)(x)=f(1/x)/x.
\]
We have 
\[
\int_0^\infty \abs{(\Theta f)(x)}^2\dd\mu^{\#}(x)
=
\int_0^\infty \abs{f(1/x)}^2\frac{\dd\mu^{\#}(x)}{x^2}
=
\int_0^\infty \abs{f(x)}^2\dd\mu(x),
\]
by the change of variable $x\mapsto1/x$, and so $\Theta$ is unitary. 
By a direct calculation with the same change of variable, we have
\[
\jap{G_{\mu^{\#}}\Theta f,\Theta g}_{\mu^{\#}}
=
\jap{G_\mu f,g}_\mu, \quad \forall f,g\in L^2(\mu),
\]
and so 
\begin{equation}
\Theta^*G_{\mu^{\#}}\Theta =G_\mu.
\label{c2a}
\end{equation}
Furthermore, we clearly have 
\begin{equation}
\Theta\theta_\mu=\1_{\mu^{\#}}.
\label{c3aa}
\end{equation}
By Lemma~\ref{lma.c5},  $\1_{\mu^{\#}}$ is a cyclic element for $G_{\mu^{\#}}$. Thus $\theta_\mu$ is a cyclic element for $G_\mu$. 
\end{proof}

Let us define the spectral measure $\rho$ corresponding to $\theta_\mu$:
\[
\jap{\varphi(G_\mu)\theta_\mu,\theta_\mu}_\mu
=
\int_0^\infty \varphi(\lambda)\dd\rho(\lambda). 
\]
Clearly, $\rho\in\calM_\finite$. Thus we have the spectral map
\[
\Omega^{\#}: \calM_\cofinite\to\calM_\finite, \quad \mu\mapsto\rho.
\]
\begin{theorem}\label{thm.c5}
We have 
\[
\Omega^{\#}=\Omega\circ\#
\]
on $\calM_\cofinite$. In particular, $\Omega^{\#}$ is a bijection between $\calM_\cofinite$ and $\calM_\finite$ and the inverse is given by 
\[
(\Omega^{\#})^{-1}=\#\circ\Omega.
\]
\end{theorem}
\begin{proof}
Using \eqref{c2a} and \eqref{c3aa}, for all bounded continuous test functions $\varphi$ we find
\[
\jap{\varphi(G_\mu)\theta_\mu,\theta_\mu}_\mu
=
\jap{\varphi(G_{\mu^{\#}})\1_{\mu^{\#}},\1_{\mu^{\#}}}_{\mu^{\#}},
\]
and so 
\[
\int_0^\infty \varphi(\lambda)\dd\rho(\lambda)=
\int_0^\infty \varphi(\lambda)\dd\sigma(\lambda), 
\]
where $\rho=\Omega^{\#}(\mu)$ and $\sigma=\Omega(\mu^{\#})$. Thus, we get 
$\sigma=\rho$ and so $\Omega^{\#}=\Omega\circ\#$. 
\end{proof}

\section{Hankel operators in $L^2(\bbR_+)$}
\label{sec.a}

Throughout this section, the measure $\mu$ is in $\calM$ and $\Gamma_\mu$ is the Hankel operator from Definition~\ref{def:1}.

\subsection{The case $\mu\in\calM_\finite$}\label{s:cv}

Since $G_\mu = L_\mu^* L_\mu$, we can write the polar decomposition of $L_\mu$ as 
\begin{align}\label{a7-a}
L_\mu = V G_\mu^{1/2}, 
\end{align}
where $V:L^2(\mu)\to L^2$ is an isometry (recall that  $\Ker G_\mu=\{0\}$). Since $\Gamma_\mu 
=L_\mu L_\mu^*$, we can write 
\begin{align}\label{a7}
\Gamma_\mu = VG_\mu^{1/2}G_\mu^{1/2}V^*=V G_\mu V^*, 
\end{align}
and we conclude that $\Ker V^*=\Ker \Gamma_\mu$ and
\begin{align*}
\Ran V = (\Ker \Gamma_\mu)^\perp = \overline{ \Ran \Gamma_\mu} .
\end{align*}
Considering $V$ as a unitary operator from $L^2(\mu)$ to $(\Ker \Gamma_\mu)^\perp$ we conclude from 
\eqref{a7} that 
\[
\Gamma_\mu^\perp
\quad\text{ and }\quad
G_\mu 
\]
are unitary equivalent. Continuing this line of reasoning, we obtain the following lemma.

\begin{lemma}\label{l:cyc02}
Assume $\mu\in\calM_\finite$. 
The spectrum of $\Gamma_\mu^\perp$ is simple and 
$v_\mu:=V\1_\mu$ is a cyclic vector with $\|v_\mu\|^2=\mu(\bbR_+)$. 
The spectral measure of $\Gamma_\mu^\perp$ with respect to $v_\mu$ coincides with the spectral measure of $G_\mu$ with respect to $\1_\mu$. 
The element $v_\mu$ satisfies the identity \eqref{z12} for all $g\in\calD$. 
\end{lemma}
\begin{proof}
By Lemma~\ref{lma.c5} the vector  $\1_\mu$ is a cyclic vector for 
$G_\mu$, so by the unitary equivalence \eqref{a7} the vector $V\1_\mu$ is  cyclic  for $\Gamma_\mu^\perp$. Since $V$ is an isometry, we have $\|v_\mu\|^2=\|\1_\mu\|_\mu^2 =\mu(\bbR_+)$. For any bounded continuous function $\varphi$ we have 
\begin{align*}
\jap{\varphi(\Gamma_\mu)v_\mu, v_\mu} 
=\jap{V\varphi(G_\mu) V^* V \1_\mu, V \1_\mu } 
=\jap{\varphi(G_\mu)\1_\mu, \1_\mu}_\mu, 
\end{align*}
and so the corresponding spectral measures coincide. 

Let us prove \eqref{z12}. By \eqref{a7}, we have the identity 
\[
\Gamma_\mu^{1/2} = V G_\mu^{1/2} V^*. 
\]
Using this, we can write 
\begin{align*}
\jap{v_\mu, \Gamma_\mu^{1/2}g} & =  \jap{v_\mu, V G_\mu^{1/2} V^* g} 
=  \jap{V^* v_\mu,  G_\mu^{1/2} V^* g}_\mu   \\
&=  \jap{V^* v_\mu,  L g}_\mu =  \jap{\1_\mu,  L g}_\mu  \\
&=  \int_0^\infty \left(\int_{0}^\infty\overline{g(s)} e^{-st} \dd s\right) \dd\mu(t) \\
&=  \int_{0}^\infty \overline{g(s)} \left(\int_{0}^\infty e^{-st} \dd \mu(t)\right) \dd s 
=\int_0^\infty h_\mu(s)\overline{g(s)}\dd s,
\end{align*}
as required. 
\end{proof}
In particular, this proves Theorems~\ref{prp.z1} (in the case $\mu\in\calM_\finite$) and~\ref{prp.z2}(i).

\begin{proof}[Concluding the proof of Theorem~\ref{thm.z2}]
The main Theorem~\ref{thm.z2} follows directly from Theorem~\ref{thm.c7}(ii) and Lemma~\ref{l:cyc02}.
\end{proof}

\subsection{The case $\mu\in\calM_\cofinite$}
Assume $\mu\in\calM_\cofinite$. 
Again, we have the polar decomposition \eqref{a7-a} and the unitary equivalence \eqref{a7}. 
Let us define
\[
w_\mu=V\theta_\mu.
\]
Since $V$ is an isometry, the norms of $w_\mu$ and $\theta_\mu$ coincide. 
Since $\theta_\mu$ is a cyclic element for $G_\mu$, it follows that $w_\mu$ is a cyclic element for $\Gamma_\mu^\perp$. Moreover, we have
\[
\jap{\varphi(\Gamma_\mu)w_\mu,w_\mu}
=
\jap{\varphi(G_\mu)\theta_\mu,\theta_\mu}_\mu,
\]
and so the spectral measure of $\Gamma_\mu^\perp$ with respect to the element $w_\mu$ coincides with the spectral measure $\rho$ of $G_\mu$ with respect to the element $\theta_\mu$.

\begin{lemma}\label{l:cyc03}
Assume $\mu\in\calM_\cofinite$. 
The spectrum of $\Gamma_\mu^\perp$ is simple and 
$w_\mu:=V\theta_\mu$ is a cyclic vector with 
\[
\|w_\mu\|^2=\int_0^\infty\frac{\dd\mu(x)}{x^2}.
\]
The spectral measure of $\Gamma_\mu^\perp$ with respect to $w_\mu$ coincides with the spectral measure of $G_\mu$ with respect to $\theta_\mu$. 
The element $w_\mu$ satisfies the identity \eqref{z12a} for all $g\in C^\infty_{\comp}(\bbR_+)$. 
\end{lemma}
\begin{proof}
By \eqref{a7}, we have 
\[
\Gamma_\mu^{1/2}=VG_\mu^{1/2}V^*.
\]
Using this, for any $g\in C^\infty_{\comp}(\bbR_+)$ we find
\begin{align*}
\jap{w_\mu,\Gamma_\mu^{1/2}g}
&=
\jap{V\theta_\mu,VG_\mu^{1/2}V^*g}
=
\jap{\theta_\mu,G_\mu^{1/2}V^*g}_\mu
=\jap{\theta_\mu,L g}_\mu
\\
&=\int_0^\infty \left\{\int_0^\infty \overline{g(t)}e^{-tx}\dd t\right\}\frac{\dd\mu(x)}{x}
=\int_0^\infty \left\{\int_0^\infty \frac{e^{-tx}}{x}\dd \mu(x)\right\}\overline{g(t)}\dd t.
\end{align*}
Finally, it is easy to see that 
\[
\int_0^\infty \frac{e^{-tx}}{x}\dd \mu(x)=\int_0^\infty h_\mu(t+s)\dd s, \quad t>0,
\]
which completes the proof of  \eqref{z12a}. 
\end{proof}
In particular, this proves Theorems~\ref{prp.z1} (in the case $\mu\in\calM_\cofinite$) and~\ref{prp.z2}(ii).

\begin{proof}[Concluding the proof of Theorem~\ref{thm.z2a}]
Theorem~\ref{thm.z2a} follows directly from Theorem~\ref{thm.c5} and Lemma~\ref{l:cyc03}.
\end{proof}

\section{Proofs of properties of $\Omega$}\label{sec.p}

\begin{proof}[Proof of Theorem~\ref{thm.z6}]
Let $\Omega(\mu)=\sigma$ and $\Omega(\tau\mu)=\sigma'$. We need to check that $\sigma'=\tau\sigma_\tau$. From \eqref{c1} it is easy to see that
\[
G_{\tau\mu}=\tau G_{\mu}
\]
and therefore 
\begin{align*}
\int_0^\infty \varphi(x)\dd\sigma'(x)
&=\jap{\varphi(G_{\tau\mu})\1_{\tau\mu},\1_{\tau\mu}}_{\tau\mu}
=\tau\jap{\varphi(\tau G_{\mu})\1_{\mu},\1_{\mu}}_{\mu}
\\
&=\tau \int_0^\infty \varphi(\tau x)\dd\sigma(x)
=\tau\int_0^\infty \varphi(x)\dd\sigma_{\tau}(x).
\end{align*}
It follows that $\sigma'=\tau\sigma_\tau$ as required. 

Let $\Omega(\mu)=\sigma$ and $\Omega(\mu_\tau)=\sigma'$. We need to check that $\sigma'=\sigma_{1/\tau}$. Consider the unitary map
\[
U_\tau:L^2(\mu)\to L^2(\mu_\tau), \quad
(U_\tau f)(x)=f(x/\tau). 
\]
Clearly, $U_\tau\1_{\mu}=\1_{\mu_\tau}$. 
It is straightforward to check that 
\[
G_{\mu_\tau}U_\tau=\tau^{-1}U_\tau G_\mu,
\]
and so 
\[
G_{\mu_\tau}=\tau^{-1}U_\tau G_\mu U_\tau^*.
\]
From here we find
\begin{align*}
\int_0^\infty \varphi(x)\dd\sigma'(x)
&=\jap{\varphi(G_{\mu_\tau})\1_{\mu_\tau},\1_{\mu_\tau}}_{\mu_\tau}
=\jap{\varphi(\tau^{-1}U_\tau G_\mu U_\tau^*)\1_{\mu_\tau},\1_{\mu_\tau}}_{\mu_\tau}
\\
&=\jap{U_\tau \varphi(\tau^{-1}G_\mu )U_\tau^*\1_{\mu_\tau},\1_{\mu_\tau}}_{\mu_\tau}
=\jap{\varphi(\tau^{-1}G_\mu )U_\tau^*\1_{\mu_\tau},U_\tau^*\1_{\mu_\tau}}_{\mu}
\\
&=\jap{\varphi(\tau^{-1}G_\mu )\1_{\mu},\1_{\mu}}_{\mu}
=\int_0^\infty \varphi(\tau^{-1}x)\dd\sigma(x)
=\int_0^\infty \varphi(x)\dd\sigma_{1/\tau}(x),
\end{align*}
which proves that $\sigma'=\sigma_{1/\tau}$.
\end{proof}

\begin{proof}[Proof of Theorem~\ref{thm.z7}]
We have 
\[
\Tr\Gamma_\mu=\Tr L^*L=\int_0^\infty \int_0^\infty e^{-2tx}\dd\mu(x)\dd t=\frac12\int_0^\infty\frac{\dd\mu(x)}{x},
\]
which proves the equivalence of (i) and (ii) and the first equality in \eqref{z15}. 

Next, we already know that $\Gamma_\mu^\perp$ and $G_\mu$ are unitarily equivalent. By Theorem~\ref{thm.c7}, the operators $G_\mu$ and $X_\sigma$ are unitarily equivalent. Thus, we have 
\[
\Tr\Gamma_\mu=\Tr\Gamma_\mu^\perp=\Tr G_\mu=\Tr X_\sigma.
\]
Clearly, $\Tr X_\sigma$ is finite if and only if $\sigma$ is supported on a summable sequence of points $\{a_k\}_{k=1}^\infty$ and in this case 
\[
\Tr X_\sigma=\sum_{k=1}^\infty a_k.
\]
This proves the equivalence of (i) and (iii) and the second equality in \eqref{z15}. 
\end{proof}

\begin{proof}[Proof of Theorem~\ref{thm.z8}]
This is a minor variation of the proof of the previous theorem. We have
\[
\Tr\Gamma_\mu^2
=\Tr G_\mu^2
=\int_0^\infty \int_0^\infty \frac{\dd\mu(t)\dd\mu(s)}{(t+s)^2},
\]
which proves the equivalence of (i) and (ii) and the first equality in \eqref{z17}. Next, by the same argument as in the proof of the previous theorem,
\[
\Tr\Gamma_\mu^2=\Tr X_\sigma^2=\sum_{k=1}^\infty a_k^2,
\]
which proves the equivalence of (i) and (iii) and the second equality in \eqref{z17}.
\end{proof}

\section{Proof of continuity of $\Omega$}\label{sec.ct}
Here we prove Theorem~\ref{thm.z9}.

\subsection{The strong convergence of $\Gamma_{\mu_n}$}
We recall some facts from \cite{PuTreil}. We denote
\[
C^\infty_{\comp,0}(\bbR_+)
=\{f\in C^\infty_\comp(\bbR_+): {\textstyle \int_0^\infty f(t)\dd t=0}\}
=\{g': g\in C^\infty_\comp(\bbR_+)\}.
\]
By \cite[Theorem~2.4(ii)]{PuTreil}, for any $\mu\in\calM$, the set $C^\infty_{\comp,0}(\bbR_+)$ is dense in $\Dom\Gamma_\mu$ in the graph norm of $\Gamma_\mu$ (i.e. this set is a \emph{core} for $\Gamma_\mu$). Denoting $\dd\nu(x)=x\dd\mu(x)$, for any $g\in C^\infty_\comp(\bbR_+)$ we have by \cite[Lemma~5.1]{PuTreil}
\begin{equation}
\Gamma_\mu g'=\Gamma_\nu g \quad\text{ and }\quad 
(\Gamma_\nu g)(t)=\int_0^\infty h_\nu(t+s)g(s)\dd s. 
\label{eq:ct1}
\end{equation}
\begin{lemma}\label{lma.ct1}
Let $\{\mu_n\}_{n=1}^\infty$, $\mu$ be measures in $\calM$, and suppose that $h_{\mu_n}(t)\to h_\mu(t)$ for every $t>0$ as $n\to\infty$. Then for any $f\in C^\infty_{\comp,0}(\bbR_+)$ we have 
\begin{equation}
\norm{\Gamma_{\mu_n}f-\Gamma_{\mu}f}\to0, \quad n\to\infty.
\label{eq:ct3}
\end{equation}
Moreover, $\Gamma_{\mu_n}\to\Gamma_\mu$ in the sense of strong resolvent convergence. 
\end{lemma}
\begin{proof}
Let us prove \eqref{eq:ct3}. By \eqref{eq:ct1}, this is equivalent to proving that 
\begin{equation}
\norm{\Gamma_{\nu_n}g-\Gamma_{\nu}g}\to0, \quad n\to\infty,
\label{eq:ct2}
\end{equation}
for any $g\in C^\infty_{\comp}(\bbR_+)$, where $\dd\nu(x)=x\dd\mu(x)$ and $\dd\nu_n(x)=x\dd\mu_n(x)$. We will prove this by using dominated convergence and the explicit integral of \eqref{eq:ct1}. Assume $\supp g\subset[\alpha,\beta]$ with some $0<\alpha<\beta<\infty$. We first need an elementary bound on the kernel $h_\nu$. For $t\geq\alpha$ we have 
\begin{align*}
th_\nu(t)=\int_0^\infty e^{-tx}tx\, \dd\mu(x)
=
\int_0^\infty e^{-tx/2}e^{-tx/2}tx\, \dd\mu(x)
\leq 
\int_0^\infty e^{-\alpha x/2}e^{-tx/2}tx\, \dd\mu(x).
\end{align*}
Using that $te^{-t/2}\leq 2$ for $t\geq0$, from here we find 
\[
h_\nu(t)\leq \frac{2}{t}\int_0^\infty e^{-\alpha x/2}\dd\mu(x)=\frac{2h_\mu(\alpha/2)}{t}.
\]
Since $h_{\mu_n}(\alpha/2)\to h_{\mu}(\alpha/2)$ by our assumptions, the sequence $h_{\mu_n}(\alpha/2)$ is bounded. Thus, we have a uniform bound
\[
h_{\nu_n}(t)\leq \frac{C_\alpha}{t}, \quad t\geq\alpha
\]
with a constant $C_\alpha$ independent of $n$. From here we deduce a uniform upper bound for $\Gamma_{\nu_n}g$: 
\begin{equation}
\abs{(\Gamma_{\nu_n}g)(t)}
\leq 
\int_\alpha^\beta h_{\nu_n}(t+s)\abs{g(s)}\dd s
\leq
\int_{\alpha}^\beta \frac{C_\alpha}{t+s}\abs{g(s)}\dd s
\leq 
\frac{C_g}{t+\alpha}, \quad t>0.
\label{eq:ct6}
\end{equation}
Of course, the right hand side here is in $L^2$. 

Next, from the convergence $h_{\mu_n}(t)\to h_{\mu}(t)$ we easily deduce the convergence $h_{\nu_n}(t)\to h_{\nu}(t)$ as $n\to\infty$ locally uniformly on any interval separated away from the origin. From here using \eqref{eq:ct1}, we obtain the pointwise convergence 
\[
(\Gamma_{\nu_n}g)(t)\to(\Gamma_\nu g)(t), \quad n\to\infty
\]
for any $t>0$. By dominated convergence from here and \eqref{eq:ct6} we obtain \eqref{eq:ct2}. 

Since $C^\infty_{\comp,0}(\bbR_+)$ is a common core for $\Gamma_{\mu_n}$ and $\Gamma_\mu$, the strong resolvent convergence $\Gamma_{\mu_n}\to\Gamma_\mu$ follows by Theorem~VIII.25(a) of \cite{RS1}.
\end{proof}

\subsection{The strong convergence of $v_{\mu_n}$}

We first isolate an elementary abstract statement which is probably well-known.
\begin{lemma}
Let $v_n$, $v$ be vectors in a Hilbert space and let $P$ be an orthogonal projection in the same space. Assume that 
\begin{equation}
Pv_n\to v \quad\text{ weakly and}\quad \norm{v_n}\to\norm{v}
\label{eq:ct7}
\end{equation}
as $n\to\infty$. 
Then $\norm{v_n-v}\to0$ as $n\to\infty$. 
\end{lemma}
\begin{proof}
Let us first prove that $\norm{Pv_n-v}\to0$. We recall that from the weak convergence $Pv_n\to v$ it follows that 
\[
\norm{v}\leq\liminf_{n\to\infty}\norm{Pv_n}
\]
and the convergence is strong if and only if we have equality here. Since $\norm{Pv_n}\leq \norm{v_n}$, using the second assumption in \eqref{eq:ct7} we find 
\[
\norm{v}\leq\liminf_{n\to\infty}\norm{Pv_n}\leq \liminf_{n\to\infty}\norm{v_n}=\norm{v}
\]
and thus we have a chain of equalities here. In particular, the convergence $Pv_n\to v$ is strong. Now write 
\[
\norm{v_n}^2=\norm{Pv_n}^2+\norm{(I-P)v_n}^2.
\]
By the previous step we have $\norm{Pv_n}\to\norm{v}$. It follows that $\norm{(I-P)v_n}\to0$. Thus we obtain the strong convergence $v_n\to v$ as claimed. 
\end{proof}

\begin{lemma}\label{lma.ct2}
Let $\mu_n\to\mu$ weakly in $\calM_\finite$. Then $\norm{v_{\mu_n}- v_\mu}\to0$ as $n\to\infty$. 
\end{lemma}
\begin{proof}
\emph{Step 1:}
Let $f\in C^\infty_{\comp,0}(\bbR_+)$. 
By Lemma~\ref{lma.ct1}, we have the convergence \eqref{eq:ct3}. Let us prove that 
\begin{equation}
\norm{\Gamma_{\mu_n}^{1/2}f-\Gamma_\mu^{1/2}f}\to0, \quad n\to\infty.
\label{eq:ct4}
\end{equation}
This is a standard argument. We use the integral representation 
\[
\Gamma^{1/2}=\frac1\pi\int_0^\infty (\Gamma+t)^{-1}\Gamma \frac{\dd t}{\sqrt{t}}
\]
valid for any positive semi-definite self-adjoint operator $\Gamma$. Applying this to $\Gamma=\Gamma_{\mu_n}$, we find 
\[
\Gamma_{\mu_n}^{1/2}f=\frac1\pi\int_0^\infty (\Gamma_{\mu_n}+t)^{-1}\Gamma_{\mu_n}f \frac{\dd t}{\sqrt{t}}
\]
where the integral converges absolutely and uniformly in $n$ because 
\[
\norm{(\Gamma_{\mu_n}+t)^{-1}\Gamma_{\mu_n}f}\leq \norm{\Gamma_{\mu_n}f}/t 
\]
(we can use this estimate for $t\geq1$) and also 
\[
\norm{(\Gamma_{\mu_n}+t)^{-1}\Gamma_{\mu_n}f}
\leq
\norm{(\Gamma_{\mu_n}+t)^{-1}\Gamma_{\mu_n}}\norm{f}\leq \norm{f}
\]
(we can use this estimate for $0<t<1$).
From here and \eqref{eq:ct3} we obtain \eqref{eq:ct4} by a version of the dominated convergence theorem. 

\emph{Step 2:}
We first note that using \eqref{z9}, 
\begin{equation}
\norm{v_{\mu_n}}^2=\mu_n(\bbR_+)\to\mu(\bbR_+)=\norm{v_{\mu}}^2
\label{eq:ct8}
\end{equation}
as $n\to\infty$. In particular, the norms $\norm{v_{\mu_n}}$ are uniformly bounded. 

Next, for any $f\in C^\infty_{\comp,0}(\bbR_+)$ we have
\[
\int_0^\infty h_{\mu_n}(t) \overline{f(t)}\dd t\to \int_0^\infty h_{\mu}(t) \overline{f(t)}\dd t
\]
as $n\to\infty$. By \eqref{z12}, it follows that 
\[
\jap{v_{\mu_n},\Gamma_{\mu_n}^{1/2}f}\to \jap{v_{\mu},\Gamma_{\mu}^{1/2}f}
\]
as $n\to\infty$. Using Step 1 and the uniform boundedness of $v_{\mu_n}$, from here we find 
\[
\jap{v_{\mu_n},\Gamma_{\mu}^{1/2}f}\to \jap{v_{\mu},\Gamma_{\mu}^{1/2}f}
\]
as $n\to\infty$ for any $f\in C^\infty_{\comp,0}(\bbR_+)$. Since $C^\infty_{\comp,0}(\bbR_+)$ is a core for $\Gamma_\mu$, we obtain that $v_{\mu_n}\to v_\mu$ weakly on a dense set of elements in $\overline{\Ran \Gamma_\mu}$. Let $P$ be an orthogonal projection onto $\overline{\Ran \Gamma_\mu}$. We obtain that 
\[
Pv_{\mu_n}\to Pv_\mu=v_\mu
\]
weakly as $n\to\infty$. Combining this with \eqref{eq:ct8} and using the previous lemma, we obtain the strong convergence $v_{\mu_n}\to v_\mu$. 
\end{proof}

\subsection{Proof of Theorem~\ref{thm.z9}}
Let $\sigma=\Omega(\mu)$ and $\sigma_n=\Omega(\mu_n)$, and let $\varphi$ be any bounded continuous function on $[0,\infty)$. By Lemma~\ref{lma.ct1}, we have the strong resolvent convergence $\Gamma_{\mu_n}\to\Gamma_\mu$. By \cite[Theorem VIII.20]{RS1} it follows that 
\[
\norm{\varphi(\Gamma_{\mu_n})v_\mu-\varphi(\Gamma_{\mu})v_\mu}\to0
\]
as $n\to\infty$.
Using this and Lemma~\ref{lma.ct2}, we find that 
\[
\jap{\varphi(\Gamma_{\mu_n})v_{\mu_n},v_{\mu_n}}
\to
\jap{\varphi(\Gamma_\mu)v_\mu,v_\mu}
\]
as $n\to\infty$. By the definition of the map $\Omega$, this means
\[
\int_0^\infty \varphi(x)\dd\sigma_n(x)\to\int_0^\infty \varphi(x)\dd\sigma(x)
\]
as $n\to\infty$. Thus, $\sigma_n\to\sigma$ weakly in $\calM_\finite$. The proof of Theorem~\ref{thm.z9} is complete.
\qed

\section{Examples}\label{sec.exa}
Here we give two examples of Hankel operators with spectral measures that can be computed explicitly. In the first example the measure $\mu$ is finite and in the second one it is co-finite.

\subsection{The Mehler operator}
Consider the Hankel operator with the kernel function
\[
h_\mu(t)=\frac1{t+2}=\int_0^\infty e^{-t x}e^{-2x}\dd x,
\]
corresponding to the measure $\dd\mu(x)=e^{-2x}\dd x$. Clearly, $\mu\in\calM_\finite$. 
This operator has been diagonalised by Mehler \cite[formula (16)]{Mehler} in 1881. He proved that the ``generalised eigenfunctions'' $P_{-\frac12+ik}(1+t)$ (here $P$ is the Legendre function) satisfy the eigenvalue equation
\[
\int_0^\infty\frac{P_{-\frac12+ik}(1+s)}{2+t+s}\dd s
=\pi\sech(\pi k)P_{-\frac12+ik}(1+t), \quad t>0,\quad k>0,
\]
where $\sech(x)=1/\cosh(x)$. 
Following Yafaev \cite{Ya-commutator}, we will call $\Gamma_\mu$ \emph{Mehler's operator}. 
More precisely, it can be shown \cite{Ya-commutator} that the operator
\begin{align*}
\Phi: L^2(\bbR_+)&\to L^2(\bbR_+)
\\
(\Phi f)(k)&=\sqrt{k\tanh \pi k}\int_0^\infty P_{-\frac12+ik}(t+1)f(t)\dd t, 
\end{align*}
is unitary and intertwines Mehler's operator $\Gamma_\mu$ with the operator of multiplication by 
$\pi\sech(\pi k)$: 
\[
(\Phi\Gamma_\mu f)(k)=\pi\sech(\pi k)(\Phi f)(k), \quad k>0.
\]
In order to determine the corresponding spectral measure $\sigma$, we proceed from formula \eqref{z11}: 
\[
\int_0^\infty \varphi(\lambda)\dd\sigma(\lambda)
=\jap{\varphi(\Gamma_\mu)\Gamma_\mu\delta_0,\delta_0}
=\int_0^\infty \varphi(\pi\sech \pi k)\pi \sech\pi k \abs{(\Phi\delta_0)(k)}^2\dd k.
\]
Observe that 
\[
(\Phi\delta_0)(k)=
\sqrt{k\tanh \pi k}\,  P_{-\frac12+ik}(1)=\sqrt{k\tanh \pi k},
\]
because $P_\nu(1)=1$ for all $\nu$.
Thus, we find
\[
\int_0^\infty \varphi(\lambda)\dd\sigma(\lambda)
= \int_0^\infty \varphi(\pi\sech \pi k)(\pi \sech\pi k \tanh \pi k) k \dd k.
\]
Changing the integration variable to $\lambda=\pi \sech\pi k$, we find that 
\[
\int_0^\infty \varphi(\lambda)\dd\sigma(\lambda)
=
\frac1{\pi^2}\int_0^{\pi}\varphi(\lambda)\sech^{-1}(\lambda/\pi)\dd\lambda,
\]
and so finally the spectral measure $\dd\sigma$ is supported on $[0,\pi]$, with 
\[
\dd\sigma(\lambda)=\frac1{\pi^2}\sech^{-1}(\lambda/\pi)\dd\lambda.
\]

From here, using our main result Theorem~\ref{thm.z2}, we get the following interesting corollary. Consider the Hankel operator $\Gamma_\sigma$ with the kernel function 
\[
h_\sigma(t)=\int_0^\infty e^{-t\lambda}\dd\sigma(\lambda)=\frac1\pi\int_0^1e^{-\pi t\lambda}\sech^{-1}\lambda\dd\lambda. 
\]
Then this operator has the spectral measure $\dd\mu(\lambda)=e^{-2\lambda}\dd\lambda$. In particular, the spectrum of $\Gamma_\sigma$ is simple, purely absolutely continuous and coincides with the interval $[0,\infty)$.

We do not know how to express the kernel $h_\sigma$ in terms of known special functions, nor are we aware of the operator $\Gamma_\sigma$ appearing in the literature. A calculation shows that
\[
h_\sigma(t)=\frac1{\pi^2}\frac{\log t}{t}(1+o(1)), \quad t\to\infty,
\]
and in particular the estimate $h_\sigma(t)\leq C/t$ fails at infinity, in agreement with the fact that $\Gamma_\sigma$ is unbounded.

\subsection{The Rosenblum operator}
Consider the Hankel operator with the kernel function
\[
h_\mu(t)=e^{-t/2}/t=\int_{1/2}^\infty e^{-tx}\dd x,
\]
corresponding to the measure $\dd\mu(x)=\chi_{(1/2,\infty)}(x)\dd x$. Clearly, $\mu\in\calM_\cofinite$. 
This operator has been diagonalised by Rosenblum \cite{Ros}, see also \cite{L1,L2,Ya-commutator}

Let $K_\nu$ be the modified Bessel function of the third kind \cite[Section 7.2.2]{BE}. It can be shown that the operator 
\begin{align*}
\Phi: L^2(\bbR_+)&\to L^2(\bbR_+)
\\
(\Phi f)(k)&=\frac1\pi \sqrt{2k\sinh(\pi k)}\int_0^\infty t^{-1/2}K_{ik}(t/2)f(t)\dd t,\quad k>0,
\end{align*}
is unitary and intertwines Rosenblum's operator $\Gamma_\mu$ with the operator of multiplication by $\pi\sech(\pi k)$. In order to determine the corresponding measure $\rho$, we proceed from formula \eqref{z11a}: 
\[
\int_0^\infty \varphi(\lambda)\dd\rho(\lambda)
=\jap{\varphi(\Gamma_\mu)\Gamma_\mu\1,\1}
=\int_0^\infty \varphi(\pi\sech \pi k)\pi \sech\pi k \abs{(\Phi \1)(k)}^2\dd k.
\]
Using \cite[6.561(16)]{GR}
\begin{align*}
(\Phi \1)(k)=\frac1\pi\sqrt{2k\sinh(\pi k)}\int_0^\infty K_{ik}(t/2)t^{-1/2}\dd t
=
\frac1{2\pi}\sqrt{2k\sinh(\pi k)}\Abs{\Gamma\left(\tfrac14+i\tfrac{k}{2}\right)}^2.
\end{align*}
Thus, we find 
\[
\int_0^\infty \varphi(\lambda)\dd\rho(\lambda)
=
\frac1{2\pi}\int_0^\infty \varphi(\pi\sech \pi k)k\tanh(\pi k)\Abs{\Gamma\left(\tfrac14+i\tfrac{k}{2}\right)}^4\dd k.
\]
Changing the integration variable to $\lambda=\pi \sech\pi k$, we find that 
\[
\int_0^\infty \varphi(\lambda)\dd\rho(\lambda)
=
\frac1{2\pi^2}\int_0^{\pi}\varphi(\lambda)\frac{k(\lambda)}{\lambda}\Abs{\Gamma\left(\tfrac14+i\tfrac{k(\lambda)}{2}\right)}^4
\dd\lambda,
\]
where 
\[
k(\lambda)=\frac1\pi \sech^{-1}(\lambda/\pi). 
\]
Thus, $\rho$ is supported on $[0,\pi]$ with 
\[
\dd\rho(\lambda)=
\frac1{2\pi^2}\frac{k(\lambda)}{\lambda}\Abs{\Gamma\left(\tfrac14+i\tfrac{k(\lambda)}{2}\right)}^4
\dd\lambda.
\]

\appendix

\section{Background in control theory }\label{sec.aa5}
The ideas behind the proof of the main Theorem~\ref{thm.z2} originate from control theory. 
The link between control theory and the spectral theory of Hankel operators was first used by 
R.~Ober \cite{Ober1,Ober2} and further developed in \cite{Treil,MPT}.

Here we explain the control theory context. 

\subsection{Dynamical systems and balanced realizations of non-negative Hankel 
operators}\label{s:BalReal}
Consider a one-dimensional dynamical system $(A, b, c)$
\begin{align*}
\left\{ 
\begin{array}{ll}
x'(t)&= -A x(t) + u(t) b\\
y(t)&= \jap{x(t), c}\ci\calH 
\end{array}
\right.
\end{align*}
with the input $u$ and output $y$. Here $A$ is a bounded operator in a Hilbert space $\calH$, 
$x=x(t)$ is an $\calH$-valued function, $b$ and $c$ are vectors in $\calH$ and $u$, $y$ are scalar-valued functions. The initial condition $x(0)=0$ is usually imposed.

The function 
\begin{align}\label{e:real01}
h(t):= \langle e^{-tA}b,c \rangle\ci{\calH}, \qquad t\ge 0,
\end{align}
is called the \emph{impulse response} of the system $(A, b, c)$; formally this is the output $y$ for input $u$ being the $\delta$-function at $0$. The Laplace transform of $h$ is called the \emph{transfer function} of the system.  

Given a function $h=h(t)$ we say that a system $(A, b, c)$ is a \emph{realization} of $h$ if 
\eqref{e:real01} holds. In this paper $h$ is the kernel of an integral Hankel operator $\Gamma_h$, and in this case we say that $(A, b, c)$ is a \emph{realization} of the Hankel operator $\Gamma_h$. 

A system $(A, b, c)$ is called \emph{controllable} if $b$ is a cyclic vector for $A$ and 
\emph{observable} if $c$ is a cyclic vector for $A^*$. A system that is both observable and 
controllable is called \emph{minimal}.

Let us also assume that  the  controllability and observability Gramians $W\ti{c}$ and $W\ti{o}$ exist. This means that the integrals
\begin{align*}
W\ti{c} &:= \int_0^\infty e^{-tA} \left(\jap{\bigcdot,b}b\right) e^{-tA^*}  \dd t= \int_0^\infty 
\jap{\bigcdot,e^{-tA}b}e^{-tA}b\,  \dd t, \\  
W\ti{o} &:= \int_0^\infty e^{-tA^*} \left(\jap{\bigcdot,c}c\right) e^{-tA}  \dd t =  \int_0^\infty 
\jap{\bigcdot,e^{-tA^*}c}e^{-tA^*}c \, \dd t, 
\end{align*}
converge in the weak operator topology to bounded operators in $\calH$. It is easy to see that in this case the system is controllable if and only if $\Ker W\ti{c}=\{0\}$, 
and it is observable if and only if $\Ker W\ti{o}=\{0\}$.   

The system is called \emph{balanced} if 
it is minimal and $W\ti{c}=W\ti{o}$; in this case we will write simply $W$ for these Gramians. 

It was shown in \cite{MPT} that 
\begin{enumerate}
\item For a balanced system $(A,b,c)$ 
and  $h$ given by \eqref{e:real01},  the modulus 
$|\Gamma_h|:=(\Gamma_h^*\Gamma_h)^{1/2}$ of the corresponding bounded Hankel operator $\Gamma_h$, restricted to $(\Ker\Gamma_h)^\perp$, is unitarily equivalent to $W$.  

\item If $\Gamma_h = \Gamma_h^* \ge 0$ and $(A, b, c)$ is its 
balanced realization, then $A=A^*>0$ and $b=c$; for the rest of this Appendix we  restrict 
ourselves to this case. 

\item  If the Gramian
\begin{align}\label{e:Gram02}
W= \int_0^\infty \jap{\bigcdot,e^{-tA}b}e^{-tA}b\,  \dd t,
\end{align}
of a system (triple) $(A, b, b)$ with $A=A^*>0$ exists (i.e.~if the integral converges in the weak operator topology to a bounded operator), then $W$ is the unique solution of the Lyapunov equation
\begin{align}\label{e:Lyap03}
AW+WA=\jap{\bigcdot,b}b.   
\end{align}  
\item If bounded operators $A=A^*>0$ and $W$ satisfy \eqref{e:Lyap03}, then the integral on the right hand side of \eqref{e:Gram02} converges in the weak operator topology and $W$ coincides with this integral. This means, in particular, that in this case $W$ is the unique solution of the Lyapunov equation \eqref{e:Lyap03}. This also implies that if $\Ker W=\{0\}$, then $b$ must be a cyclic vector for $A$. 

Note that here one does not assume anything about Hankel operators or balanced realizations.  We are just given operators $A$, $W$, and the vector $b$; it might happen that $\Ker W\ne\{0\}$, so the system is not balanced because it is not minimal. But if 
$\Ker W=\{0\}$,  the system is a balanced realization of some non-negative Hankel operator.

\end{enumerate}

We should mention that $A$ and $W$ enter the above Lyapunov equation \eqref{e:Lyap03} 
symmetrically. This implies that if the kernel of one of the operators $A$, $W$ is trivial, then the vector $b$ is cyclic for the other one. This symmetry was exploited 
in \cite{MPT} to show that the spectrum of any non-negative integral Hankel operator $\Gamma_h$ admitting a balanced realization with bounded generator $A$ is simple. (As a consequence, non-negative Hankel operators with multiple spectrum do not admit balanced realizations; example: the Carleman operator.) This symmetry also plays an important role in this paper.

\subsection{Extension to the unbounded operators}
In this paper we extend the above reasoning to the case of unbounded generators $A$ and unbounded Gramians $W$. 

First we notice that while for general measures $\mu\in\calM$  the function $\1_\mu$ might fail to 
be in $L^2(\mu)$, we see that for operator  $X_\mu$ defined in Section~\ref{sec.c}, the element 
$e^{-tX_\mu}\1_\mu \in L^2(\mu)$,  so we can write 
\begin{align*}
h_\mu(t)= \jap{e^{-tX_\mu}\1_\mu, \1_\mu}_\mu = \jap{e^{-tX_\mu/2}\1_\mu, 
e^{-tX_\mu/2}\1_\mu}_\mu . 
\end{align*}
Thus we can say that the system $(X_\mu, \1_\mu, 
\1_\mu)$ is a realization of a Hankel operator $\Gamma_\mu$ (i.e.~of the function $h=h_\mu$). 
Moreover, the identity \eqref{c3a}, appropriately interpreted, means that $G_\mu$ is the Gramian of 
$(X_\mu, \1_\mu, \1_\mu)$ and since  $\Ker G_\mu=\{0\}$, we can 
say that $(X_\mu, \1_\mu, \1_\mu)$ is a \emph{balanced} realization. 
Thus we can say that any non-negative 
Hankel operator $\Gamma_\mu$ admits a  balanced representation, if all objects are correctly 
interpreted. 

The operators $X_\mu$ and $G_\mu$ defined in Section~\ref{sec.c} satisfy 
(at least formally) the Lyapunov equation 
\begin{align}\label{e:Lyap04}
X_\mu G_\mu + G_\mu X_\mu = \jap{\bigcdot, \1_\mu}\1_\mu  .  
\end{align}
For measures $\mu\in\calM_\finite$ we gave rigorous interpretation of the equation 
\eqref{e:Lyap04}, and proved in  Section \ref{sec.c}  that in this case $G_\mu$ is 
the Gramian of the system $(X_\mu, \1_\mu, \1_\mu)$ and $X_\mu$ is the Gramian of the system 
$(G_\mu, \1_\mu, \1_\mu)$, see Lemma \ref{lma.c4}.  

We should mention that while formally $X_\mu$ and $G_\mu$  participate symmetrically in the 
Lyapunov equation \eqref{e:Lyap04}, there essential asymmetry between these operators, even for 
$\mu\in\calM_{\finite}$. 
Namely, the fact that  $G_\mu$ is the Gramian of the system $(X_\mu, \1_\mu, \1_\mu)$, i.e.~the identity \eqref{c3a}, is simply 
a restatement of the identity $G_\mu=L_\mu^* L_\mu$. But the fact that for $\mu\in\calM_{\finite}$ the operator $X_\mu$ is 
the Gramian of the system $(G_\mu, \1_\mu, \1_\mu)$, i.e.~identity \eqref{c4a}, is highly non-trivial, and its proof requires some hard work, see Section 
\ref{sec.c}.  
And since trivially $\Ker X_\mu=\{0\}$, this identity implies that $\1_\mu$ is a cyclic vector for  $G_\mu$, which is at the crux of our construction.

\subsection{Asymmetry between $X_\mu$ and $G_\mu$ for general measures}
The above mentioned asymmetry between $X_\mu$ and $G_\mu$ is even more dramatic in the case of general measures $\mu$ (i.e. not finite or co-finite). Namely, for an arbitrary measure  $\mu\in\calM$   the operator $G_\mu$ is still the Gramian of the system $(X_\mu, \1_\mu, \1_\mu)$, i.e.~the identity \eqref{c3a} still holds as 
an identity for the quadratic form on an appropriate dense set; as before, it is just a restatement of the identity  $G_\mu=L_\mu^* L_\mu$.

However for the general measures $\mu\in\calM$  the operator $G_\mu$ can have spectrum of
multiplicity $2$. The simplest example is again the Carleman operator, i.e.~the Hankel operator
$\Gamma_\mu$  where the measure $\mu$ is the Lebesgue measure $\dd x$ on $\bbR_+$; this operator
has the spectrum of multiplicity $2$, and the corresponding operator $G_\mu$ coincides with the
Carleman operator itself.

But if $G_\mu$ does not have a simple spectrum, the operator  $X_\mu$ cannot be the Gramian of the system $(G_\mu, \1_\mu, \1_\mu)$ in any reasonable sense.  

Indeed, as we discussed above,  $e^{-tG_\mu}\1_\mu \in L^2(\mu)$ for all $t>0$, so the left hand 
side of \eqref{c4a} is well defined. Thus, if  
the identity \eqref{c4a} holds for a reasonable class of test functions $f$, the fact that $\Ker 
X_\mu=\{0\}$ would imply that  $\Span\{e^{-tG_\mu}\1_\mu:t>0\}$ is dense in $L^2(\mu)$, i.e.~that the operator $G_\mu$ has a simple spectrum. 
This contradicts our initial assumption, so $X_\mu$ cannot be the Gramian of the system $(G_\mu, \1_\mu, \1_\mu)$.

\end{document}